\documentclass[
a4paper,
11pt,
twoside,
]{article}
\usepackage[utf8]{inputenc}
\usepackage[T1]{fontenc}
\usepackage[sc]{mathpazo}
\usepackage[english]{babel}
\usepackage{geometry}
\usepackage{amsmath,amsfonts,amssymb,amsthm}
\usepackage[final,colorlinks]{hyperref}
\usepackage{graphicx,subcaption}
\usepackage{mathtools}
\usepackage[dvipsnames]{xcolor}
\usepackage{url}
\usepackage[square,numbers]{natbib}
\definecolor{darkblue}{cmyk}{1,0,0,0.8}
\definecolor{darkred}{cmyk}{0,1,0,0.7}
\hypersetup{anchorcolor=black,
  citecolor=darkblue, filecolor=darkblue,
  menucolor=darkblue,urlcolor=darkblue,linkcolor=darkblue}
\usepackage{pdflscape}
\usepackage{etoolbox,paralist}
\usepackage{fancyhdr}
\usepackage{comment}
\usepackage{lastpage}
\usepackage{ifdraft}
\usepackage{hyphenat}
\ifdraft{\usepackage{showlabels}}{}
\usepackage{enumitem}
\usepackage[multiple]{footmisc}
\usepackage{siunitx}
\usepackage[mathlines]{lineno}

\newcommand*\patchAmsMathEnvironmentForLineno[1]{%
  \expandafter\let\csname old#1\expandafter\endcsname\csname #1\endcsname
  \expandafter\let\csname oldend#1\expandafter\endcsname\csname end#1\endcsname
  \renewenvironment{#1}%
  {\linenomath\csname old#1\endcsname}%
  {\csname oldend#1\endcsname\endlinenomath}}%
\newcommand*\patchBothAmsMathEnvironmentsForLineno[1]{%
  \patchAmsMathEnvironmentForLineno{#1}%
  \patchAmsMathEnvironmentForLineno{#1*}}%
\patchBothAmsMathEnvironmentsForLineno{equation}%
\patchBothAmsMathEnvironmentsForLineno{align}%
\patchBothAmsMathEnvironmentsForLineno{flalign}%
\patchBothAmsMathEnvironmentsForLineno{alignat}%
\patchBothAmsMathEnvironmentsForLineno{gather}%
\patchBothAmsMathEnvironmentsForLineno{multline}%

\DeclareMathOperator{\lip}{Lip}
\DeclareMathOperator{\ev}{Ev}
\DeclareMathOperator{\rg}{rg}
\DeclareMathOperator{\Lin}{Lin}

\newcommand{\dd}{\mathop{}\!\mathrm{d}}
\newcommand{\lm}{{\ell_{\max}}}

\theoremstyle{plain}
\newtheorem{theorem}{Theorem}[section]
\newtheorem{lemma}[theorem]{Lemma}
\newtheorem{proposition}[theorem]{Proposition}
\newtheorem{corollary}[theorem]{Corollary}
\newtheorem{definition}[theorem]{Definition}
\newtheorem{assumption}[theorem]{Assumption}

\numberwithin{equation}{section}
\numberwithin{table}{section}
\numberwithin{figure}{section}

\newcommand{\R}{\mathbb{R}}

\DeclareMathOperator{\B}{B}
\DeclareMathOperator{\err}{Err}

\renewcommand{\i}{\mathrm{i}}
\newcommand{\e}{\mathrm{e}}

\newcommand{\gfde}{G_\mathrm{FDE}}

\DeclareRobustCommand{\red}[1]{\textcolor{black}{#1}{}}
\fancypagestyle{mypagestyle}{%
  \fancyhf{}%
  \fancyhead[LO,RE]{\scriptsize Boundary-value problems with state-dependent delays}%
  \fancyhead[RO,LE]{\scriptsize }%
  \fancyfoot[LE,RO]{\scriptsize \thepage\,/\,\pageref*{LastPage}}%
}
\title{Spectral element methods for boundary-value problems of functional differential equations}
\author{
Alessia And\`o$^{1}$, Jan Sieber$^{2}$
\\[.5em]
\small $^{1}$CDLab -- Computational Dynamics Laboratory\\[-.2em]
\small Department of Mathematics, Computer Science and Physics -- University of Udine\\[-.2em]
\small via delle scienze 206, 33100 Udine, Italy\\[-.2em]
\small \texttt{alessia.ando@uniud.it}\\[.5em]
\small $^{2}$Department of Mathematics and Statistics -- University of Exeter\\[-.2em]
\small North Park Road, Exeter EX4 4QF, U.\,K.\\[-.2em]
\small \texttt{j.sieber@exeter.ac.uk}
}
\date{\today}
\begin{document}
\clearpage
\maketitle
\thispagestyle{empty}
\begin{abstract}\noindent
 We prove convergence of the spectral element method for piecewise polynomial collocation applied to periodic boundary value problems for functional differential equations. In particular, we prove that the numerical collocation solution approximates the true solution with accuracy of order $\mathrm{e}^{-\eta m}$ for some $\eta>0$ and increasing degree $m$ of the polynomials, \red{provided that the true solution is analytical. This includes} a case that is common in applications: differential equations where the right-hand side depends on a finite number of delayed arguments with parametric delays and real analytic coefficients. For state-dependent delays the spectral element method also converges under mild regularity assumptions, \red{although analyticity of the solution cannot be trivially inferred from analyticity of the coefficients. 
In order to extend our convergence results to this case, we} introduce the concept of extended local Lipschitz continuity of the right-hand side.\end{abstract}


\section{Introduction}
Periodic orbits are one of the possible attractors or invariant sets one wants to track when studying long-term dynamics of dynamical systems depending on several (varying or uncertain) parameters \citep{D07}. To find periodic orbits systematically, piecewise orthogonal collocation for boundary-value problems (BVPs) is used in continuation-based software packages that have been developed for the stability and bifurcation analysis of ordinary differential equations, such as \textsc{MatCont} \citep{DGK03}, \textsc{AUTO} \citep{D07} or \textsc{coco} \citep{DS13}. Software libraries such as \textsc{DDE-Biftool} \citep{ELR02,ddebiftoolmanual} or \textsc{knut} \citep{RS07} perform a subset of the possible bifurcation analysis for functional differential equations (FDEs, also called delay differential equations/DDEs). These libraries are used by engineers and scientists in application areas ranging from population or traffic modelling \citep{diekmann2010daphnia,gedeon2022operon,OWES10} to analysis of chatter in milling \citep{mohammadi2022chatter}. The current user interfaces of both libraries are able to handle FDEs defined by an arbitrary number of possibly state-dependent discrete delays in the arguments of the right-hand side of the differential equations.  However, the underlying 
collocation method, originally developed by \citet{ascher1981collocation} and applied to FDEs by \citet{ELHR01}, is applicable to more general types of delays.

\paragraph{Periodic BVPs for FDEs} We consider periodic orbits of dynamical systems with delayed arguments, modelled by FDEs, which are differential equations of the form
\begin{align}
  \label{intro:fde}
  y'(t)&=\gfde(y_t,p)\mbox{,}
\end{align}
with dependent variable $y:\R\to\R^{n_y}$, problem parameters $p\in\R^{n_p}$, and the continuous nonlinear functional
  $\gfde:C^0([-\tau_{\max},0];\R^{n_y})\times\R^{n_p}\to\R^{n_y}$
as right-hand side. We use $C^0(I;\R^n)$ for the space
of continuous functions from the interval $I$ into $\R^n$ (and, correspondingly, $C^\ell$ for $\ell$ times continuously differentiable functions), $C^\ell_\pi$ for spaces of functions that are $C^\ell$ on $\R$ with period $1$, and the notation for the time shift of a function by time $t$
\begin{align}\label{intro:timeshift}
    y_t=y(t+(\cdot)).
\end{align} 
When searching for periodic orbits of \eqref{intro:fde} of unknown period $T$, one rescales time and adds $n_p+1$ affine constraints to formulate
a ``square'' periodic FDE BVP
\begin{align}\label{res:bvp}
  y'(t)&=TG_\mathrm{FDE}(y(t+(\cdot)/T),p)=:G(y_t,\mu)\mbox{\quad ($t\in[0,1]$),}&
  0&=R_\mathrm{aff}[y,\mu],
\end{align}
for the unknowns $y\in C^1_\pi(\R^{n_y})$ and $\mu=(T,p)\in\R^{n_\mu}$ ($n_\mu=n_p+1$). The $y$ component of the unknown is now a periodic function of known period $1$, while the period of the orbit is explicitly included as a parameter in the unknown $\mu$. The new functional
$G:C^0_\pi(\R^{n_y})\times
\R^{n_\mu}\to\R^{n_y}$ is continuous, and the affine map $R_\mathrm{aff}:C^0_\pi(\R^{n_y})\times\R^{n_\mu}\to\R^{n_\mu}$ defines $n_\mu$ bounded affine constraints.
 Our arguments below assume that \eqref{res:bvp} has a solution $(y^*,\mu^*)\in C^1_\pi(\R^{n_y})\times\R^{n_\mu}$.

\paragraph{Collocation discretization} Collocation discretizes a BVP such as \eqref{res:bvp} by dividing the base interval $[0,1]$ into $L$ subintervals, approximating the solution $y^*$ by $y^{m,L}$, which is a polynomial of degree $m$ on each subinterval. One imposes the differential equation on $(y^{m,L},\mu^{m,L})$ at $m$ \emph{collocation points} in each subinterval, along with continuity and periodicity conditions on $y^{m,L}$. The \emph{finite element method} (or strategy, FEM) keeps the polynomial degree $m$ fixed and lets $L$ go to infinity, thus decreasing the length of subintervals proportional to $1/L$. With this strategy one can hope for uniform convergence with errors $(y^*-y^{m,L},\mu^*-\mu^{m,L})$ of order $L^{-m}$ for $L\to\infty$ under suitable assumptions on $y^*$  and $G$. In contrast, the \emph{spectral element method} (SEM) keeps $L$ fixed and lets the degree $m$ go to infinity. This strategy requires a careful choice of collocation points, but it can potentially achieve exponential accuracy under favorable circumstances (in our case real analyticity of $y^*$): one may hope for errors of order $\e^{-\eta m}$ for a positive $\eta$.  The exponent $\eta$ depends on the mesh of subintervals and on the domain in the complex plane in which $y^*$ has an analytic extension \citep{boyd01,tref12}.

\paragraph{Existing convergence results} The question of convergence of collocation for periodic FDE BVPs and the FEM and SEM strategies has long been open and is not yet entirely settled. Engelborghs and Doedel (the authors of \textsc{DDE-Biftool} and \textsc{AUTO}) analysed convergence for FDE BVPs in \citep{ED02}, but limited their analysis to linear constant-delay right-hand sides and assumed that the period $T$ is known, which is generally not the case for periodic orbits in applications. Much later, Maset  provided a general, abstract framework for the convergence of discretizations of  FDE BVPs \citep{mas15NM}. However, Maset's framework requires regularity of the right-hand side $G$ in the sense of classical Fr{\'e}chet differentiability (continuous differentiability with respect to arguments $y_t$ and $\mu$). This requirement is in general not satisfied for FDE \eqref{res:bvp}. \citet{andoSIAM2020} overcame this fundamental limitation of Maset's theory to prove convergence of the FEM strategy for \eqref{res:bvp} with constant delays. They exploited that for constant delay one may assume that $\gfde$ is continuously differentiable $m$ times. Thus, the only point at which these problems violate Maset's assumptions is the lack of continuous differentiability of $G$ with respect to $T$ in \eqref{res:bvp}. The convergence result \citep{andoSIAM2020} combines careful estimates of the dependence of $G$ on $T$ with Maset's theory \citep{mas15SINA2,mas15SINA1,mas15NM} to arrive at the expected error bounds proportional to $L^{-m}$ for $L\to\infty$. If the delays are permitted to depend on the state, $\gfde$ is not continuously differentiable, making the framework of Maset inapplicable. In general, $\gfde$ is only continuously differentiable in a weaker sense, a concept which \citet{asFEM} call \emph{mild differentiability}: $\gfde$ is $\ell$ times continuously differentiable in its arguments, if $y$ is in $C^\ell$ (see Definition~\ref{def:mild} below for full details). Using these weaker concepts \citet{asFEM} show that the FEM strategy converges with error proportional to $L^{-m}$ for $L\to\infty$, if  $\gfde$ is $m$ times mildly differentiable. 

\paragraph{Convergence of SEM strategy} Our work explores convergence of the SEM strategy for \eqref{res:bvp}, which had been left as an open problem even in the case of constant delays. There are two obstacles for proving convergence of SEM. First, the constants entering the error bounds in \citep{andoSIAM2020,asFEM} grow with the degree $m$ of the piecewise polynomial $y^{m,L}$. This is not a problem for the FEM, which keeps $m$ constant, but it causes an instability for SEM, which increases $m$. The second problem is that there is no concept of mild analyticity for right-hand sides $\gfde$ of FDEs. \citet{MPN14} showed that periodic orbits of FDEs can be non-analytic even if all coefficients in $\gfde$ ``look analytic'', because the series expansion \red{of the periodic orbit} diverges in some points. Our illustrative example \eqref{eq:quadratic}, with $\gfde(y,p)=-y(-p-y(0)-y(0)^2)$ for different parameters $p$, fits into the example class considered  by \citet{MPN14}.
As observed in \citep{andoSIAM2020}, its techniques for proving convergence do not lend themselves easily to proving convergence of SEM \citep{breda2005pseudospectral,T96}. One reason is the choice of working with spaces of non-periodic functions, as will be clarified in Section \ref{sec:convergence}.

\section{Main results}
\label{sec:result}

\subsection{Periodic BVPs for FDEs and mild differentiability}
\label{sec:res:mild}
Nonlinearities such as $G_\mathrm{FDE}$ or $G$ in \eqref{res:bvp} require modified concepts of continuous differentiability and local Lipschitz continuity to describe their regularity in the case of state-dependent delays, which we call \emph{mild differentiability} (identical to \cite{asFEM}) and a new \emph{extended local Lipschitz continuity} (needed for $DG$). 
\begin{definition}[Mild differentiability and extended local Lipschitz continuity]\label{def:mild}
    A map $G:C^0(I;\R^{n_u})\to\R^{n_G}$ is called $\lm$ times mildly differentiable if for all $\ell\in\{0,\ldots,\lm\}$
    \begin{compactenum}
        \item The map $G\vert_{C^\ell}$ is $C^\ell$, and
        \item the map $C^\ell\times C^\ell\ni (u,\delta_u)\mapsto D^\ell G(u)(\delta_u)^\ell\in\R^{n_G}$ is continuously extendable to $C^\ell\times C^{\ell-1}$, \red{where $(\delta_u)^\ell$ is the product-like $\ell$-tuple $\delta_u\cdots\delta_u$ of arguments of the multilinear map $D^\ell G(u)$}.
    \end{compactenum}
    Let $G:C^0(I;\R^{n_u})\to\R^{n_G}$ be mildly differentiable at least once. We say that
  $DG$ satisfies an extended local Lipschitz condition in
  $u\in C^2$ if there exist a constant $C>0$ and a radius $r>0$
  such that
\begin{align}
  \label{def:extlip:eq}
  |DG(u+v)y-DG(u)y|\leq C\|v\|_1\|y\|_1\mbox{\quad for all $v\in \B^1_r(0)\subset C^1$, $y\in C^1$.}
\end{align}
\end{definition}
In \eqref{def:extlip:eq} we use the notations \begin{align*}
    \|v\|_\ell&=\sup_{t\in I}\left\{|v(t)|,|v'(t)|,\ldots,|v^{(\ell)}(t)|\right\},&
    \B_r^\ell(w)&=\{v\in C^\ell:\|v-w\|_\ell\leq r\}
\end{align*} 
for the supremum norm in $C^\ell$ and for a ball of radius $r$ around $w$. 
\red{\citet{MNP94} introduced the concept of mild differentiability under the name \emph{almost Fr{\'e}chet differentiability}. It is based on the observation that all FDEs with state-dependent delays and smooth coefficients have this property, and has been used to develop general theory for FDEs  with state-dependent delays in other papers \cite{HKWW06,K03,S12,asFEM}.}
Classical continuous differentiability of $G$ of order $2$ implies local Lipschitz continuity of $DG$. However, mild differentiability of $G$ of order $2$ only implies a weaker (``\emph{mild}'') local Lipschitz condition for $DG$ in $u$, which looks like \eqref{def:extlip:eq} but only applies to $v$ in a sufficiently small ball $\B_r^2(0)$ (in contrast to $\B_r^1(0)$ as required in \eqref{def:extlip:eq}). \red{We will comment on the calculus of mild and extended Lipschitz conditions such as \eqref{def:extlip:eq} in Section~\ref{sec:assumptions}.} Both parts of Definition~\ref{def:mild} are also applicable to functionals depending on parameters, such as $\gfde$, by applying the definition to the functional $G:C^0(I;\R^{n_y+n_p})=C^0(I;\R^{n_y})\times C^0(I;\R^{n_p})\ni (y,p)\mapsto\gfde(y,p(t_0))\in\R^{n_y}$ for an arbitrary time $t_0\in I$ (so, we treat the parameter $p$ as a function, which happens to be constant).
\begin{assumption}[Assumptions on the BVP]\label{res:ass}\
  \begin{compactenum}
  \item \label{res:ass:mdiff} \textup{(Mild differentiability)} The
    right-hand side $G_\mathrm{FDE}$ in the DDE
    \eqref{res:bvp} is mildly differentiable to order
    $\lm\geq 2$. 
  \item \label{res:ass:existence} \textup{(Existence of solution)} BVP \eqref{res:bvp} has a solution $(y^*,\mu^*)=(y^*,T^*,p^*)$.
  \item \label{res:ass:linear:inv} \textup{(Linear well-posedness)}
    The BVP \eqref{res:bvp}, linearized in $(y^*,T^*,p^*)$,
    has only the trivial solution $(\delta^y,\delta^T,\delta^p)=0$.
    \item \label{res:ass:ext:llip} \textup{(Extended local Lipschitz condition)} $D\gfde$ satisfies the extended local Lipschitz condition \eqref{def:extlip:eq} in $(y^*_t,T^*,p^*)$ for all $t\in[0,1]$.
  \end{compactenum}
\end{assumption}
For a solution $(y^*,T^*,p^*)$ of BVP~\eqref{res:bvp} with $\lm$ times mildly differentiable $\gfde$, the component $y^*$ is in $C^{\lm+1}$ (see \cite{asFEM}), such that the linearizations used in points \ref{res:ass:linear:inv} and \ref{res:ass:ext:llip} exist.
\subsection{Convergence of solution of discretized BVP}
\label{sec:res:collocation}
For a polynomial collocation discretization the unknown function is a $1$-periodic
continuous piecewise polynomial $y^m$ on a mesh of $L$ subintervals given as
$0=t_0<\ldots<t_L=1$. More precisely, $y^m$ is a polynomial of degree $m$ on
$[t_{i-1},t_i]$ for all $i=1,\ldots,L$ in each of its $n_y$
components, $y^m$ is continuous, and $y^m(0)=y^m(1)$, such that we can extend $y^m$ to a function in $C^0_\pi$. The strategy we consider
for adjusting approximation quality is a spectral-element approach keeping the mesh size $L$ fixed,
and considering the limit $m\to\infty$. Thus, we do not include superscript $L$ for the variable $y^m$. Additional unknowns
are the parameters $\mu^m=(T^m,p^m)$, resulting in $n_y mL+n_p+1=n_ymL+n_\mu$ unknowns
overall.  The system of algebraic equations,
\begin{align}\label{res:bvp:disc}
  0&=(y^m)'(t_{i,j})-G(y^m_{t_{i,j}},\mu^m),&
    0&=R_\mathrm{aff}[y^m,\mu^m]
\end{align}
for $1\leq i\leq L$, $1\leq j\leq m$, evaluates the FDE at the
collocation points $t_{i,j}=t_{i-1}+(t_i-t_{i-1})t_{\mathrm{c},j}$,
where the points $t_{\mathrm{c},j}$ are the $m$ collocation points for
degree $m-1$ on the interval $[0,1]$ for a sequence of orthogonal
polynomials (e.g. Gauss-Legendre or Chebyshev nodes, as used in \textsc{DDE-Biftool}). We recall that, for a given set of $m$ nodes in the interval $[0,1]$, the Lebesgue constant is defined as 
\begin{equation*}
\Lambda_{m}=\max_{t\in[0,1]}\sum_{j=1}^m|l_{m,j}(t)|,
\end{equation*}
where $\{l_{m,1},\ldots,l_{m,m}\}$ is the Lagrange basis for the nodes $(t_{\mathrm{c},j})_{j=1}^m$. Our assumption on the discretization is that the Lebesgue constant, and, hence, the norm of the interpolation projection, diverges only slowly for functions in $C^0_\pi$.
\begin{assumption}[Slowly divergent interpolation constant]\label{ass:lebesgue}
The Lebesgue constant $\Lambda_m$ of the collocation point sets $(t_{\mathrm{c},j})_{j=1}^m$ on $[0,1]$ satisfies
  \begin{align*}
    \frac{\Lambda_m}{m}\to 0\mbox{\quad for $m\to\infty$.}
  \end{align*}
\end{assumption}
Spectral element strategies are most useful when $y^*$ is analytic, since in that case they promise exponential convergence, provided that certain bounds on all derivatives of
$y^*$ are satisfied. Assumption~\ref{ass:lebesgue} is quite restrictive, but includes the Chebyshev and Gauss-Legendre node sequences \cite{boyd01,tref12}. In particular,  the arguments by Trefethen \cite{tref12} imply that exponential accuracy (\emph{geometric convergence}) of interpolation \red{for real analytic functions} follows from Assumption~\ref{ass:lebesgue} (see Section~\ref{sec:analytic} for more details):
\begin{proposition}[Exponential accuracy of interpolation]\label{thm:geom:app}
If the nodes $(t_{\mathrm{c},j})_{j=1}^m$ have  a slowly diverging Lebesgue constant $\Lambda_m$ according to Assumption~\ref{ass:lebesgue} and $v$ is analytic  on a complex neighborhood of $[0,1]$, then there exist constants $\eta>0$ and $C_\mathrm{sp}>0$ such that
\begin{align}
 \label{eq:approx:geom}
  \sup_{t\in[0,1]}|v(t)-v^m(t)|\leq C_\mathrm{sp}\e^{-\eta m}\mbox{\quad for all $m\geq1$,}
\end{align}
where $v^m$ is the unique polynomial of degree $m-1$ satisfying $v^m(t_{\mathrm{c},j})=v(t_{\mathrm{c},j})$ for all $j=1,\ldots,m$.
\end{proposition}  
We can now state the main convergence theorem for discretized
BVP \eqref{res:bvp:disc}.
\begin{theorem}[Convergence of discretization]\label{res:thm:conv}
  Under Assumption~\ref{res:ass} on FDE BVP~\eqref{res:bvp} and Assumption~\ref{ass:lebesgue} on node sequence $(t_{\mathrm{c},j})_{j=1}^m$, \red{for all sufficiently large $m$} the discretized BVP
  \eqref{res:bvp:disc} with mesh $(t_i)_{i=0}^L$ has a 
  unique solution $x^m=(y^m,T^m,p^m)$ \red{in a ball
  $\B^{0,1}_{r_m}(x^*)$, where $x^*=(y^*,T^*,p^*)$ is the solution of BVP \eqref{res:bvp} and the radius $r_m$ depends on $m$.} 
  The solution satisfies
  \begin{align}\label{res:thm:conv:est}
    \|x^m-x^*\|_{0,1}:=\max\{\|y^m-y^*\|_{0,1},|T^m-T^*|,|p^m-p^*|\}=O\left(\frac{\Lambda_m}{m^{\lm}}\right).
  \end{align}
  Moreover, if $y^*$ has an analytic extension to a complex neighborhood of $[0,1]$, then
    \begin{align}\label{res:thm:conv:exp}
    \|x^m-x^*\|_{0,1}=O\left(\e^{-\eta m}\right),
  \end{align}
  for some $\eta>0$ which depends on $x^*$ and the mesh $(t_i)_{i=0}^L$.
\end{theorem}
\noindent
\red{In \eqref{res:thm:conv:est}  $\lm$ is the degree of mild differentiability of $\gfde$, and $\Lambda_m$ is the Lebesgue constant of node sequence $(t_{\mathrm{c},j})$.}
The norm $\|y^m-y^*\|_{0,1}$ in \eqref{res:thm:conv:est} is the Lipschitz norm $\|z\|_{0,1}=\max\{|z(t)|, |z(t)-z(s)|/|t-s|:t,s\in[0,1],s\neq t\}$. \red{The radius $r_m$ of the ball in which we claim uniqueness of the discretized solution shrinks for increasing $m$ to compensate for the increasing Lebesgue constant $\Lambda_m$.}

We observe that the discretized solution converges faster than with order $\lm-1$ if \red{$\lm\geq 2$ is finite. The bound $\lm\geq2$ in Assumption~\ref{res:ass} is indeed needed throughout for the stability argument (Lemma~\ref{thm:stability}), so \eqref{res:thm:conv:est} is not guaranteed for $\lm= 1$.}
If the solution of the original BVP \eqref{res:bvp} is analytic in a neighborhood of $[0,1]$ the discretization error decreases exponentially with increasing degree $m$. Theorem~\ref{res:thm:conv} implies that one could choose $L=1$ (so, a single mesh interval). However, for analytic solutions $y^*$ that have singularities near the real interval $[0,1]$ the constants involved in the convergence, such as $\eta$ in \eqref{res:thm:conv:exp}, can be improved by choosing a mesh with $L>1$ if the mesh boundaries are close to the singularities of $y^*$.

The geometric convergence toward analytic solutions makes the SEM an appealing method for finding periodic orbits in FDEs. In the case of constant delays, periodic orbits of FDEs with analytic coefficients are analytic \citep{N73}. However, the problem of analyticity of periodic orbits of state-dependent FDEs is more intricate because on a periodic orbit of an FDE with a discrete state-dependent delay the delay as a function of time is periodic. For examples from this class \citet{MPN14} have demonstrated a mechanism where the periodic orbit may be $C^\infty$ but not analytic. In those cases Theorem~\ref{res:thm:conv} still implies convergence beyond all finite orders (that is, faster than $O(m^{-\ell})$ for all $\ell$). We comment on non-analyticity at the end of Section \ref{sec:analytic} and demonstrate rates of convergence numerically for an illustrative example in Section~\ref{sec:tests}.

\section{Fixed-point problem equivalent to periodic BVP and relevant notation}
\label{sec:formulation}
\paragraph{Basic notation} Firming up the notation of Section~\ref{sec:result} we use the spaces and norms
\begin{align*}
  C^{k\phantom{,1}}_\pi&\!=\{v:\mbox{\ $k\times$\,cont.\,diff.,\,} v(t)=v(t+1)\mbox{\ for all $t$}\},& 
\|v\|_{k\phantom{,1}}&\!=\max_{t\in[0,1],j\leq k}|v^{(j)}(t)|,\\
  C^{k,1}_\pi&=\{v\in C^k_\pi,\lip v^{(k)}<\infty\},& 
\|v\|_{k,1}&\!=\max\{\|v\|_k,\lip v^{(k)}\},\\
L^\infty_\pi&=\{v\mbox{\ ess.bd.,\ } v(t)=v(t+1)\mbox{\ for all $t$}\},& 
\|v\|_\infty\ &=\operatorname{ess\,sup}_{t\in[0,1]}|v(t)|.
\end{align*}
The dimension of the function's value $v(t)$ is determined by
context such that we often drop domain and codomain indicators in the
spaces. Otherwise we write, e.g., $C^{k,j}_\pi(\R^{n_y})$ or $L^\infty_\pi(\R^{n_y})$.
We will append some finite-dimensional components $v^0\in\R^{n_y}$ and $\mu\in\R^{n_\mu}$ to elements of the above function spaces later. All norms of spaces for $v$ can be trivially
extended by the finite-dimensional maximum norms of these finite-dimensional components. Hence, we also define the extended
spaces
\begin{align*}
  C^{k(,1)}_\mathrm{e}&=C^{k(,1)}_\pi\times\R^{n_y}\times\R^{n_\mu},&
  L^\infty_\mathrm{e}&=L^\infty_\pi\times\R^{n_y}\times\R^{n_\mu},
\end{align*}
and continue to use the $\|\cdot\|_{k(,1)}$ or $\|\cdot\|_\infty$
notation for their respective norms. We introduce isomorphisms \red{between spaces of increasing and decreasing regularity $k$, $J_{10}:C_\pi^{k(,1)}\to C_\pi^{k+1(,1)}$ and its inverse $J_{01}$, defined by
\begin{align}
    \label{J10def:per}
    [J_{10}v](t)&=\int_0^t v(s)\dd s+(1-t)\int_0^1 v(s)\dd s,&
    [J_{01}w](t)&=w'(t)+w(0).
\end{align}
Similarly, isomorphisms for $C^{k(,1)}_\mathrm{e}$ spaces} are $J^\mathrm{e}_{10}:C^{k(,1)}_\mathrm{e}\to C^{k+1(,1)}_\mathrm{e}$ with inverse $J^\mathrm{e}_{01}$,
\begin{align}\label{J10def}
    J^\mathrm{e}_{10}\begin{bmatrix}
        v(\cdot)\\v^0\\\mu
    \end{bmatrix}&=\begin{bmatrix}
        t\mapsto v^0+\int_0^t v(s)\dd s-t\int_0^1 v(s)\dd s\\\red{v^0+}\int_0^1v(s)\dd s\\\mu
    \end{bmatrix},&
    J^\mathrm{e}_{01}\begin{bmatrix}
        w(\cdot)\\w^0\\\mu
    \end{bmatrix}&=\begin{bmatrix}
        w'+w^0\red{-w(0)}\\w(0)\\\mu
    \end{bmatrix}.
\end{align} 
\red{The isomorphism $J^\mathrm{e}_{10}$ was called $\mathcal{L}$ in \cite{asFEM}.}
We denote a  closed ball of radius $r$ around a $u\in C^{k(,1)}_\mathrm{e}$ in the $\|\cdot\|_{k(,1)}$-norm (or in $L^\infty$ with $\|\cdot\|_\infty$-norm) by
\begin{align*}
    \B_r^{k(,1)}(u)=\left\{v\in C^{k(,1)}_\mathrm{e}:\|v-u\|_{k(,1)}\leq r\right\},&
    \B_r^\infty(u)=\left\{v\in L^\infty_\mathrm{e}:\|v-u\|_\infty\leq r\right\}.
\end{align*}
\paragraph{Infinite-dimensional fixed point problem} Similar to \cite{asFEM}, we reformulate \eqref{res:bvp} as a fixed-point problem for an operator $\Phi$ in a space of $1$-periodic
functions. We append a $n_y$-dimensional ``boundary condition'' $y(0)=y^0$ and an additional variable $y^0\in\R^{n_y}$ to the differential equation in \eqref{res:bvp}, considering the BVP
\begin{align}
    \label{gen:bvp+bc}
    y'(t)&=G(y_t,\mu)\mbox{\quad ($t\in(0,1)$),}&y(0)&=y^0,&0&=R_\mathrm{aff}[y,\mu].
\end{align} 
\red{Following \cite{asFEM}} we transform \red{BVP \eqref{gen:bvp+bc}} equivalently to the fixed-point problem
\begin{align}
  \label{gen:fixedpoint}
  x=\Phi(x)\red{=J^\mathrm{e}_{10}\circ g(x)},
\end{align}
\red{where $x$ has components $(v(\cdot),v^0,\mu)$, and the nonlinearity $g$ is} defined by
\begin{align}\label{g_ext}
  g&:C^{0,1}_\mathrm{e}\to C^0_\mathrm{e}\mbox{\quad with\quad}
  g
  \begin{pmatrix*}[l]
      v\\v^0\\\mu
    \end{pmatrix*}
  =
  \begin{bmatrix}
    t\mapsto G(v_t,\mu)\\
    v^0\\
    \mu+R_\mathrm{aff}[v,\mu]
  \end{bmatrix}.
  \end{align}
So, overall the operator $\Phi$ maps $C^{0,1}_\mathrm{e}\to C^1_\mathrm{e}$ and has the form
\begin{align}\label{Phi1}
\red{\Phi_{}\left(\begin{bmatrix}
    v(\cdot)\\v^0\\\mu\end{bmatrix}\right)}&
=\begin{bmatrix*}[l]
t\mapsto&v^0+\displaystyle\int_0^tG(v_s,\mu)\dd s-t\int_0^1G(v_s,\mu)\dd s\\[1ex]
&v^0 + \displaystyle\int_0^1G(v_s,\mu)\dd s\\[1ex]
&\mu+R_\mathrm{aff}[v,\mu]
\end{bmatrix*}.
\end{align}
\red{In \eqref{Phi1}} $v$ is $1$-periodic with $v(t)\in\R^{n_y}$, $v^0\in\R^{n_y}$,
$\mu\in\R^{n_\mu}$. 
\paragraph{Discretized fixed-point problem}
For functions $v\in C^0_\pi$ we define the interpolation
projection $P_mv$ as the unique piecewise
polynomial on mesh $(t_i)_{i=0}^L$ of degree $m-1$ where the piece on each interval $(t_{i-1},t_i)$ equals $v$ on the nodes $t_{i,j}$:
\begin{align}\label{def:interp:proj}
  P_m:&\phantom{=\ }C^0_\pi\to L^\infty_\pi,\\
  \nonumber
  P_mv&=\hat{v}\quad \begin{aligned}[t]
    &\mbox{with $\hat{v}(t_{i,j})=v(t_{i,j})$ for all $i\in\{1,\ldots,L\},j\in\{1,\ldots,m\}$,}\\
    &\mbox{and\ $\hat{v}$ is a degree $m-1$ polynomial on $[t_{i-1},t_i)$ for all $i\in\{1,\ldots,L\}$.}
\end{aligned}
\end{align}
In the name $P_m$ we do not indicate the dependence on the mesh size $L$ as we will keep this size constant, studying only the limit $m\to \infty$ in our convergence analysis. Note that $t\mapsto [P_mv](t)$ can in general not be expected to be continuous as it may have discontinuities at the mesh boundaries
$t_i$, such that the codomain of $P_m$ is $L^\infty_\pi$\red{, unless $t_{i,m}=t_i$ and $t_{i,1}=t_{i-1}$ for all $i\in\{1,\ldots,L\}$  
(e.g., Gauss-Lobatto or Chebyshev nodes of type II). Otherwise,  (e.g., for Gauss-Legendre or Chebyshev nodes of type I),  $P_mv$ is in general discontinuous at all interval boundaries $t_i$, because $P_m$ performs independent polynomial interpolation on each subinterval $[t_{i-1}, t_i)$ without enforcing inter-element continuity constraints. In practical implementations, continuity of the approximate solution $y^m$ is not ensured by $P_m$ but rather by the integration operator $J^\mathrm{e}_{10}$ in \eqref{thm:disc:fp:eq} (see also final paragraph of Section~\ref{sec:formulation} on the distinction between collocation and \emph{representation} points).}
Assumption~\ref{ass:lebesgue} and Proposition~\ref{thm:geom:app} imply the following properties for $P_m$.
\begin{corollary}[Convergence of interpolation projection with finite mesh]\label{thm:pmconv}
    Assume that the interpolation nodes $(t_{\mathrm{c},j})_{j=1}^m$ satisfy Assumption~\ref{ass:lebesgue} on the growth of the interpolation error on $[0,1]$. Then there exists a constant $C_\Lambda>0$ such that
    \begin{align}\label{thm:lebesgue:pm}
        \sup_{v\in C^0_\pi}\frac{\|v-P_mv\|_\infty}{m\|v\|_0}\leq C_\Lambda\frac{1+\Lambda_m}{m}\to0\mbox{\quad for $m\to\infty$.}    
    \end{align}
    If $v\in C^0_\pi$ has an analytic extension into a complex neighborhood of $[0,1]$ then there exist constants $\eta>0$ and $C_\mathrm{sp}>0$ such that
    \begin{align}
        \label{thm:exp:pm}
        \|v-P_mv\|_\infty\leq C_\mathrm{sp}\e^{-\eta m}\mbox{\quad for all $m\geq1$.}
    \end{align}
\end{corollary}
We extend the projection $P_m$ trivially similar to the spaces $C^{k,j}_\pi$ and $C^{k,j}_\mathrm{e}$ for the additional finite-dimensional components: 
    \begin{align}\label{def:interp:PL}
    \mathcal{P}_m:C^0_\mathrm{e}\to
       L^\infty_\mathrm{e},\qquad
    \mathcal{P}_m
    \begin{bmatrix*}[l]
      w\\w^0\\\nu
    \end{bmatrix*}
    =
       \begin{bmatrix*}[l]
         P_mw\\
         w^0\\
         \nu
       \end{bmatrix*}.
  \end{align}
We also extend the discretized BVP \eqref{res:bvp:disc} by a boundary condition and an additional variable $y^{0,m}\in\R^{n_y}$ in the same way as the FDE BVP \eqref{res:bvp}:
\begin{align}\label{gen:bvp+bc:disc}
  0&=(y^m)'(t_{i,j})-G(y^m_{t_{i,j}},\mu^m),&
  y^m(0)&=y^{0,m},&
    0&=R_\mathrm{aff}[y^m,\mu^m].
\end{align}
Considering first component of \eqref{gen:bvp+bc:disc} (the discretized FDE) as an identity for piecewise polynomials in $\rg P_m$, the extended discretized BVP \eqref{gen:bvp+bc:disc} is an identity in $\rg\mathcal{P}_m\subset L^\infty_\mathrm{e}$. Thus, we can apply $J^\mathrm{e}_{10}$ to this identity to construct the fixed point problem
\begin{align}
  \label{thm:disc:fp:eq}
    x^m&={J^\mathrm{e}_{10}}\circ \mathcal{P}_m\circ g(x^m)=:\Phi_m(x^m)\mbox{\quad for $x^m=(y^m,y^{0,m},\mu^m)$,}
\end{align}
equivalent to \eqref{gen:bvp+bc:disc},
where ${J^\mathrm{e}_{10}}:L^\infty_\mathrm{e}\to C^{0,1}_\mathrm{e}$ and $g:C^{0,1}_\mathrm{e}\to C^0_\mathrm{e}$ are defined in \eqref{J10def} and \eqref{g_ext}, and $\mathcal{P}_m$ in \eqref{def:interp:PL}. Hence, discretization of the full FDE BVP \eqref{gen:bvp+bc} corresponds to inserting projection $\mathcal{P}_m$ into the equivalent fixed-point problem \eqref{gen:fixedpoint}.

\paragraph{Numerical representation of the piecewise polynomials}
The variable for our discretized fixed-point problem is $x^m=(y^m,y^{0,m},\mu^m)\in\rg J^\mathrm{e}_{10}\mathcal{P}_m\subset C^0_\mathrm{e}$, such that $y^m$ is in the space of continuous $1$-periodic functions that are polynomials of degree $m$ on each mesh interval $[t_{i-1},t_i]$ for $i=1,\ldots,L$. In practical implementations one usually maps the variable $y^m$ into $\R^{Lmn_y}$ by choosing for each mesh interval a set of \emph{representation} nodes $(t_{\mathrm{r},j})_{j=1}^{m+1}$ with $t_{\mathrm{r},1}=0$ and $t_{\mathrm{r},m+1}=1$. The periodic piecewise polynomial $y^m$ is then represented by the values $y^m(t_{\mathrm{r},i,j})$, where $t_{\mathrm{r},i,j}=t_{i-1}+(t_i-t_{i-1})t_{\mathrm{r},j}$ such that $t_{\mathrm{r},i-1,m+1}=t_{\mathrm{r},i,1}$ for $i=1,\ldots,L$. Since $y^m$ is periodic, the value at $i=L$, $j=m+1$ stores the additional variable $y^{0,m}$. One can then obtain the value $y^m(t)$ at arbitrary times $t\in\R$ by interpolation and assuming period $1$ of the piecewise polynomial. The representation nodes $t_{\mathrm{r},j}$ can be different from the collocation nodes $t_{\mathrm{c},j}$. That would mean that the unknowns of the discrete problem are given by the values of the relevant functions at the representation nodes, while the equations need to be satisfied at the collocation nodes. Thus, implementing a solver for the discretized fixed-point problem \eqref{thm:disc:fp:eq} involves the linear map $\R^{Lmn_y}\times\R^{n_y}\times\R^{n_\mu}\to\rg J^\mathrm{e}_{10}\mathcal{P}_m$, mapping the vector $(y^m(t_{\mathrm{r},i,j}),y^{0,m})$ to $y^m(t_{i-1}+(t_i-t_{i-1})t_{\mathrm{c},j}+s)$ for $i=1,\ldots,L$, $j=1,\ldots,m$ and various times $s$ depending on the delays. The condition number of this matrix does not affect the rates of convergence in Theorem~\ref{res:thm:conv}.

\section{Regularity of right-hand sides and solutions of BVP}
\label{sec:assumptions}

The following lemma, a conclusion from \cite{asFEM}, translates mild differentiability from the functional $G$ to the map $g$, defined in \eqref{g_ext}. 
\begin{lemma}[Mild differentiability with time shift]\label{thm:diff:shift}
  Assume that the right-hand side $G:C^0_\pi(\R^{n_y})\times \R^{n_\mu}\to\R^{n_y}$ in \eqref{gen:bvp+bc} satisfies mild differentiability
  to at least order $\lm$. Then, for
  $\ell\leq\lm$ the map $g$ as a map $g:C^{\lm}_\mathrm{e}\to C^{\lm-\ell}_\mathrm{e}$, as defined in \eqref{g_ext}, is
  $\ell$ times continuously differentiable \textup{(}for $\ell=0$
  continuous\textup{)}. Furthermore, the map
    \begin{align*}
      C^{k+\ell}_\mathrm{e}\times(C^{k+\ell-1}_\mathrm{e})^\ell\ni (x,\delta^{x,1},\ldots,\delta^{x,\ell})\mapsto D^\ell g(x)\delta^{x,1}\ldots \delta^{x,\ell}\in C^k_\mathrm{e}
    \end{align*}
    is continuous for all $k\in\{0,\ldots,\lm-\ell\}$.
  \end{lemma}

This section translates the conditions listed in Assumption~\ref{res:ass} on the
infinite-dimensional BVP \eqref{gen:bvp+bc} into assumptions on the fixed point problem \eqref{gen:fixedpoint}, and comments on their consequences for the fixed point operator $\Phi$, as well as Assumption \ref{ass:lebesgue}, concerning the discretized BVP. Assumption~\ref{res:ass} includes \cite[Assumption 2.2]{asFEM}, which is needed independently of the chosen discretization strategy. \red{We will also comment on the calculus for types of local Lipschitz properties that can be expected for FDE right-hand side functionals.}
\begin{assumption}[Existence of solution, equivalent to Assumption \ref{res:ass}, Point \ref{res:ass:existence}]\label{ass:existence}The fixed-point problem \eqref{Phi1} has a solution
$x^*$ with components $(y^*,y^{0,*},\mu^*)$\textup{:} $x^*=\Phi(x^*)$.
\end{assumption}
\begin{assumption}[Mild differentiability of right-hand side $G$, equivalent to Assumption~\ref{res:ass}, Point~\ref{res:ass:mdiff}]\noindent\label{ass:mdiff}
  The right-hand side $G$ of \eqref{res:bvp} and \eqref{gen:bvp+bc} is mildly
  differentiable to order $\lm\geq 2$.
\end{assumption}
As already observed by And{\`o} and Sieber \cite{asFEM}, thanks to Lemma~\ref{thm:diff:shift} with $\ell=0$, the
$C^{\lm+1}$ regularity (smoothness) of the solution $x^*$ of the fixed-point problem \eqref{Phi1} is a direct consequence of the mild differentiability of degree $\lm$ of $G$.
\begin{corollary}[Smoothness of solution of fixed-point problem \eqref{Phi1}]\label{thm:regularity:fp}
  Let $x^*=(y^*,y^{0,*},\mu^*)$ be a solution of the fixed-point
  problem \eqref{Phi1}, and let the right-hand side $G$ satisfy mild differentiability to order $\lm$. Then
  the solution component $y^*$ is in $C^{\lm+1}_\pi$
  as a function of time, and, hence, $x^*\in C^{\lm+1}_\mathrm{e}$.
\end{corollary}
Corollary~\ref{thm:regularity:fp} allows us to conclude
that the Fr{\'e}chet derivative
\begin{align*}
  C^1_\mathrm{e}\ni x\mapsto D\Phi_{}(x)={J^\mathrm{e}_{10}}Dg(x)\in\Lin(C^1_\mathrm{e};C^1_\mathrm{e})
\end{align*}
 of the fixed point map $\Phi$ defined in \eqref{Phi1} exists and is
continuous in $x=x^*$. This follows from the continuous
differentiability of $g$ as a map from $C^1_\mathrm{e}$ to
$C^0_\mathrm{e}$ and the fact that
${J^\mathrm{e}_{10}}\in \operatorname{Lin}(C^0_\mathrm{e}; C^1_\mathrm{e})$ increases regularity by one degree.
\red{Moreover, since $D\Phi(x^*)$ is a compact linear operator on spaces $C^{k,j}_\mathrm{e}$ for all $k\geq0$ and $j=0,1$, $I-D\Phi(x^*)$ is a Fredholm operator of index 0 and the nullspace of
$I-D\Phi(x^*)$ is at most finite-dimensional}. Assumption~\ref{ass:linear:inv} requires that it is trivial, \red{which in turn implies the existence of a bounded inverse
$[I-D\Phi(x^*)]^{-1}$.}
\begin{assumption}[Well-posedness of infinite-dimensional linear problem, equivalent to Assumption~\ref{res:ass}, Point~\ref{res:ass:linear:inv}]\noindent\label{ass:linear:inv}
 The linear bounded operator $I-D\Phi_{}(x^{\ast})$ is injective on $C^1_\mathrm{e}$, that is, if $v=D\Phi(x^*)v$ and $v\in C^1_\mathrm{e}$ then $v=0$.
\end{assumption}
We may replace $C^1_\mathrm{e}$ in
Assumption~\ref{ass:linear:inv} by $C^0_\mathrm{e}$, since elements $v$ of the nullspace of $[I-D\Phi(x^*)]$ that are in
$C^0_\mathrm{e}$ are also in the image of ${J^\mathrm{e}_{10}}$ of $C^0_\mathrm{e}$, which
is in $C^1_\mathrm{e}$. 

\paragraph{\red{Motivation for extended local Lipschitz condition}} In Section~\ref{sec:convergence}, a primary additional difficulty when proving the convergence for SEM, compared to FEM, is the need to find error bounds that allow us to compensate for the growing Lebesgue constant $\Lambda_m$ as $m\to\infty$. \red{In Theorem~\ref{res:thm:conv} we have to choose a radius $r_m$ of a ball $\B_{r_m}^{0,1}(x^*)$ carefully, because this radius has to satisfy two bounds simultaneously that counteract each other. On one hand, the interpolation error $\Phi_m(x^*)-x^*$ for the true solution (consistency term, $\sim (\Lambda_m+1)/m^\lm$) has to be smaller than $r_m$, so $\Phi_m(x^*)\in \B_{r_m}^{0,1}(x^*)$, bounding $r_m$ from below. On the other hand the nonlinear terms inside this ball should not be too large, bounding $r_m$ from above. This criterion can be expressed using the modulus of continuity $r\mapsto\omega_1(r)$ of $Dg$ in $x^*$ in the appropriate norms, defined by
\begin{align*}
    \|Dg(x^*+\delta)z-Dg(x^*)z\|_0\leq \omega_1(\|\delta\|_1)\|z\|_1\mbox{\quad for small $\|\delta\|_1$}
\end{align*}
($\omega_1(0)=0$ and $\omega_1$ monotone increasing). The nonlinear terms will require $\omega_1(r_m)\lesssim1/\Lambda_m$, bounding $r_m$ from above, such that we have the two criteria
\begin{align*}
\frac{\Lambda_m+1}{m^\lm}&<r_m,&
\omega_1(r_m)&\lesssim\frac{1}{\Lambda_m}.
\end{align*} 
Continuous differentiability  of $D\gfde$ would ensure an upper bound $c_\mathrm{bd}r$ for $\omega_1(r)$, and, hence, enable us to find $r_m$ meeting both criteria for $\lm=2$ (since $\Lambda_m/m\to0$). Mild differentiability of $D\gfde$ (or second-order mild differentiability of $\gfde$) does not guarantee this, such that we need to add an explicit Lipschitz condition on $D\gfde$.}

\paragraph{\red{Mild Lipschitz properties for nonlinear functionals}}
\red{For simplicity we explain the hierarchy of Lipschitz properties  for functionals on periodic spaces, of type $G:C_\pi^0\to\R^{n_G}$. If $G$ is mildly differentiable, then for all $u\in C_\pi^1$ there exists a radius $r>0$ and a bound $C_0(u)$ such that the mild local Lipschitz condition (called \emph{almost Lipschitz} in \cite{MNP94,HKWW06})
\begin{align}\label{def:mildlip0}
    |G(v)-G(w)|&\leq  C_0(u)\|v-w\|_0\mbox{\quad holds for all $v,w\in \B_r^{0,1}(u)$} 
\end{align}
(see \cite{HKWW06}).
If $G$ is twice mildly differentiable, we can use the isomorphisms $J_{10}$ and $J_{01}$, defined in \eqref{J10def:per}, to  $C_\pi^0\times C_\pi^0\ni(z,y)\mapsto DG(J_{10}z)y$, which is mildly differentiable once, in $(z,y)\in C_\pi^1\times C_\pi^1$: there exists $r>0$, $C_0(z,y)>0$ such that for all $z_1,z_2\in \B_r^{0,1}(z)$, $y\in \B_r^{0,1}(0)$
\begin{align}\label{def:mildlip1}
        |DG(J_{10}z_1)y-DG(J_{10}z_2)y|&\leq  C_0(z,y)\|z_1-z_2\|_0\mbox{.} 
\end{align} Since $DG$ is linear in $y$, the bound $C_0(z,y)$ is linear in $y$, such that $C_0(z,y)=C_1(z)\|y\|_1$. So, setting $z=J_{01}u$, $z_1=J_{01}v$, $z_2=J_{01}w$ for arbitrary $u\in C_\pi^2$, $v,w\in \B_r^{1,1}(u)$, we obtain} 
a mild local Lipschitz
property for $DG$: 
\begin{align}
  \label{mildlip:DG}
  |DG(v)y-DG(w)y|\leq C_1(u)\|v-w\|_1\|y\|_1
\end{align}
for all $v,w\in \B^{1,1}_r(u)\subset C^{1,1}_\pi$, $y\in C^1_\pi$. \red{Since $DG(v)y$ is continuous in $y$ in the $\|\cdot\|_0$-norm and $C_\pi^{1}$ is dense in $C_\pi^{0,1}$ with respect to $\|\cdot\|_0$-norm, we can replace the factor $\|y\|_1$ by $\|y\|_{0,1}$ in \eqref{mildlip:DG}. Analogous properties can be derived for higher-order derivatives if $G$ is mildly differentiable to a higher degree.}

\paragraph{\red{Extended local Lipschitz properties for nonlinear functionals}}
 \red{There is a stronger local Lipschitz property that does not follow from mild differentiability, but is still commonly observed for right-hand sides with state-dependent delays.} 
We will need to bound the right-hand side in \eqref{mildlip:DG} for deviations $v$ and $w$ that are only small in the $C^1$-norm (not the stronger $C^{1,1}$-norm). While in the constant delay case this is a weaker requirement than $G$ being continuously differentiable twice, in the state-dependent case we need a different Lipschitz property.
\begin{definition}[Extended local Lipschitz condition, formulated for periodic functions \red{and first order}]\noindent\label{def:extlip1}
  Let $G:C^0_\pi\to\R^{n_G}$ be \red{continuous}. 
  We say that
  \red{$G$} satisfies an extended local Lipschitz condition in
  $u\in C^1_\pi$ if there exist a constant $C_0(u)>0$ and a radius $r(u)>0$
  such that
\begin{align}
  |G(u+v)-G(u)|\leq C_0(u)\|v\|_0\mbox{\quad for all $v\in \B^0_{r(u)}(0)\subset C^0_\pi$.}
\end{align}
\end{definition}
\red{This extended Lipschitz condition is around a center point $u\in C_\pi^1$, but permits deviations $v\in C_\pi^0$. This is stronger than the mild Lipschitz condition \eqref{def:mildlip0}, because \eqref{def:mildlip0} would provide a Lipschitz bound only for $v\in \B_{r(u)}^{0,1}(0)$ with some radius $r(u)$, instead of $v\in \B_{r(u)}^0(0)$. We use the same convention for parameter-dependent and multi-argument maps as explained for mild differentiability: for a functional $G:C_\pi^0(\R^{n_x})\times \R^{n_p}\to\R^{n_G}$ we apply the rules to $G_\mathrm{e}:C_\pi^0(\R^{n_x+n_p})\to\R^{n_G}$, defined by $G_\mathrm{e}((u,p))=G(u,p(0))$.}
To illustrate that assuming extended local Lipschitz continuity still permits right-hand sides of FDEs with state-dependent delays, we consider the simple example \red{of the map $\ev:C_\pi^0\times \R\ni(y,\tau)\mapsto y(\tau)\in\R$, called the evaluation map in \cite{HKWW06}, for $(u,\tau)\in C_\pi^1\times\R$. Let $v\in C_\pi^0$, $\tau,h\in\R$ be arbitrary:
\begin{multline*}
    |\ev(u+v,\tau+h)-\ev(u,\tau)|=|u(\tau+h)+v(\tau+h)-u(\tau)|\\
    \leq|u(\tau+h)-u(\tau)|+|v(\tau+h)|\leq \|u\|_1|h|+\|v\|_0\leq [\|u\|_1+1]\|(v,h)\|_0,
\end{multline*}
so the extended Lipschitz bound is $C_0(u,\tau)=\|u\|_1+1$. }

\red{Composition of functionals with extended Lipschitz property again has the extended Lipschitz property: let $F:C_\pi^0(\R^{n_x})\times\R^{n_G}\to\R^{n_F}$ and $G:C_\pi^0(\R^{n_y})\to\R^{n_G}$ have the extended Lipschitz property, and consider the functional $F\circ(\operatorname{Id}\times G)$ in a point $x\in C_\pi^1(\R^{n_x})$, $y\in C_\pi^1(\R^{n_y})$. Denote the radius around $(x,G(y))$, in which the extended Lipschitz property of $F$ holds by $r_F(x,G(y))$, the Lipschitz bound $C_F(x,G(y))$, the radius around $y$ where the extended Lipschitz property of $G$ holds by $r_G(y)$, the Lipschitz bound for $G$ by $C_G(y)$, and define $r=\min\{r_F(x,G(y)),r_G(y)\}$. Let $v\in \B_r^0(0)\subset C_\pi^0(\R^{n_x})$, $w\in \B_r^0(0)\subset C_\pi^0(\R^{n_y})$. Then
\begin{multline*}
|F(x+v,G(y+w))-F(x,G(y))|\leq C_F(x,G(y))\max\{\|v\|_0,|G(y+w)-G(y)|\}\\
\leq C_F(x,G(y))\max\{\|v\|_0,C_G(y)\|w\|_0\}
\leq C_F(x,G(y))(1+C_G(y))\|(v,w)\|_0.
\end{multline*}
Thus, the extended Lipschitz bound of $F\circ (\operatorname{Id}\times G)$ in $(x,y)$ for radius $r$ is $C_{F\circ G}=C_F(x,G(y))(1+C_G(y))$.}

\red{Since classically Lipschitz continuous functionals on $C_\pi^0$ also have the extended Lipschitz property, the algebra of all compositions (with arbitrary nesting) of Lipschitz continuous functionals and the evaluation map $\ev$ have the extended Lipschitz property, covering typical FDEs with state-dependent delays.}

\red{For the proof of Theorem~\ref{res:thm:conv} we need the extended local Lipschitz property, applied to the derivative $DG$ of a mildly differentiable functional $G$. We apply Definition~\ref{def:extlip1} to $C_\pi^0\times C_\pi^0\ni (z,y)\mapsto DG(J_{10}z)y\in\R^{n_G}$ in $(z,y)\in C_\pi^1\times C_\pi^1$ with deviations $(w,0)\in C_\pi^0\times C_\pi^0$. Since $DG(J_{10}z)y$ is linear in $y$ the radius for the Lipschitz property depends only on $z$ (we denote it by $r(z)$), and the bound depends linearly on $y$ (so, it has the form $C(z)\|y\|_1$):
\begin{align*}
    |DG(J_{10}[z+w])y-DG(J_{10}z)y|\leq C(z)\|y\|_1\|w\|_0\mbox{\quad for all $w\in \B_{r(z)}^0(0)$.}
\end{align*}
As the expression on the left is continuous in $y$ in the $\|\cdot\|_0$-norm and $C_\pi^{1}$ is dense in $C_\pi^{0,1}$ in the $\|\cdot\|_0$-norm, we can replace the factor $\|y\|_1$ by $\|y\|_{0,1}$ on the right.
For arbitrary $u\in C_\pi^2$ and $v\in C_\pi^1$ we set $w=J_{01}v\in C_\pi^0$ to formulate a natural extended Lipschitz property to impose on $DG$ for a mildly differentiable $G$.
\begin{definition}[Extended local Lipschitz condition, Definition~\ref{def:mild}, equation~\eqref{def:extlip:eq}, formulated for periodic functions]\noindent\label{def:extlip}
  Let $G:C^0_\pi\to\R^{n_G}$ be mildly differentiable once. 
  We say that
  \red{$DG$} satisfies an extended local Lipschitz condition in
  $u\in C^2_\pi$ if there exist a constant $C_1(u)>0$ and a radius $r(u)>0$
  such that
\begin{align}
  |DG(u+v)y-DG(u)y|\leq C_1(u)\|v\|_1\|y\|_{0,1}\mbox{\quad for all $v\in \B^1_{r(u)}(0)\subset C^1_\pi$, $y\in C_\pi^{0,1}$.}
\end{align}
\end{definition}
}

\paragraph{\red{Assumption of extended local Lipschitz property of derivative}} We assume that the right-hand side $G$ that is part of the nonlinearity $g$ of the fixed point
problem \eqref{gen:fixedpoint} has the extended Lipschitz property in the solution $(y^*_t,y^{0,*},\mu^*)$ for all $t$.
\begin{assumption}[Extended local Lipschitz condition in fixed point, equivalent to Assumption~\ref{res:ass}, Point~\ref{res:ass:ext:llip}]\noindent\label{ass:extlip}
  Assume that the nonlinear map $G$ in the first component of the
  nonlinearity $g$ given in \eqref{g_ext} is mildly differentiable
  once and that $DG$ has the extended local Lipschitz property \red{in Definition~\ref{def:extlip}} 
  in $(y^*_t,\mu^*)\in C^2_\pi\times\R^{n_\mu}$ for all times $t$ in the fixed point $x^*=(y^*,y^{0,*},\mu^*)$ of $\Phi$.
\end{assumption}
Thanks to the compactness of $\{y^*_t\,:t\in\R\}$ in $C^2_{\pi}$ 
there exist a uniform bound $C_{\mathrm{eL}}>0$ and a uniform radius $r_{\mathrm{eL}}>0$ such that 
\begin{align}
  \label{def:extlip:ts}
  \|Dg(x^*+v)w-Dg(x^*)w\|_0\leq C_{\mathrm{eL}}\|v\|_1\|w\|_{0,1},
\end{align}
for all $v\in C^1_\mathrm{e}$ with $\|v\|_1\leq r_\mathrm{eL}$ and
$w\in C^{0,1}_\mathrm{e}$. 

\section{Accuracy of interpolation for analytic functions} 
\label{sec:analytic}
If $v|_{[t_{i-1},t_i]}$ can be extended to an analytic function in a (complex) neighborhood of $[t_{i-1},t_i]$, then it is analytic in a \emph{Bernstein ellipse}
$$
\B_{i}=\left\{\frac{t_{i-1}+t_i}{2}+\left(\frac{t_{i-1}-t_i}{2}\right)(\cosh\eta\cos\theta-\i\sinh\eta\sin\theta)\,:\theta\in[0,2\pi],\eta\in[0,\eta_i]\right\}
$$
for some sufficiently small $\eta_i>0$ defining the shape factor $\e^{\eta_i}$ of the Bernstein ellipse. If this happens for all $1\leq i\leq L$, then $v$ has an analytic extension to $\cup_{i=1}^L\B_i$, and the best polynomial approximation of $v$ of degree $m$ on $[t_{i-1},t_i]$ (let us call it $P_{\mathrm{opt},i,m}(v)$) satisfies \cite{tref12}
\begin{align*}
    \limsup_{m\to\infty}\left\|[v-P_{\mathrm{opt},i,m}(v)]\vert_{[t_{i-1},t_i]}\right\|_\infty^{1/m}\leq \e^{-\eta_i}.
\end{align*}
Hence, for every $\eta_0<\min\{\eta_1,\ldots,\eta_L\}$ there exists a $m_0$ such that
\begin{align*}
    \left\|[v-P_{\mathrm{opt},i,m}(v)]\vert_{[t_{i-1},t_i]}\right\|_\infty&\leq \e^{-\eta_0 m}\mbox{\quad for all $m\geq m_0$ and $i=1,\ldots,L$.}
\end{align*}
By Assumption~\ref{ass:lebesgue} the Lebesgue constant $\Lambda_m$ for the collocation node sequence $(t_{\mathrm{c},j})_{j=1}^m$ satisfies that $\Lambda_m/m\to0$, and we have that \begin{align*}
    \left\|[v-v^m]_{[t_{i-1},t_i]}\right\|_\infty&\leq (1+\Lambda_m)\left\|[v-P_{\mathrm{opt},i,m}(v)]\vert_{[t_{i-1},t_i]}\right\|_\infty\\&\leq(1+\Lambda_m)\e^{-\eta_0 m}
\end{align*}
for the interpolation polynomial $v^m$ of degree $m-1$ for $v$ through the nodes $t_{i-1}+(t_i-t_{i-1})t_{\mathrm{c},j}$.
Hence, taking the maximum over all $L$ subintervals $[t_{i-1},t_i]$ and taking into account Assumption~\ref{ass:lebesgue} that $\Lambda_m=o(m)$, there exists for every $\eta\in(0,\eta_0)$ a constant $C_\mathrm{sp}$ such that
\begin{align*}
    \left\|v-P_mv\right\|_\infty\leq C_\mathrm{sp}\e^{-\eta m}\mbox{\quad for all $m\geq1$,}
\end{align*}
as claimed in Proposition~\ref{thm:geom:app}.
For specific nodes such as Chebyshev nodes, we have well-known specific estimates,
\begin{align}\label{cheb:analytic}
  \|[I-P_{m}]v\|_\infty<\frac{4M}{\e^\eta-1}\e^{-\eta m}\mbox{,\quad where $M=\max\{|v(z)|:z\in \bigcup_{i=1}^L\B_i\}$.}  
\end{align}
Despite the fact that spectral convergence for analytic functions can already be observed with $L=1$, the above construction shows that using a larger $L$ can, in principle, increase the value of $\eta$ and, correspondingly, the convergence rate. We observe this effect in our tests in Section \ref{sec:tests} (see Figure~\ref{fig:MG_conv}). Note that Legendre polynomials, equally adopted in applications, also satisfy Assumption \ref{ass:lebesgue} (see, e.g., \cite[Section 2.13]{boyd01}). Moreover, in both cases the convergence rate of the interpolation is supergeometric when $v$ is entire.
\paragraph{Analyticity of periodic orbits in DDEs}
We conclude this section with some remarks on the possibility to determine a priori whether a periodic solution is analytic. In the constant delay case, it is a consequence of the right-hand side $\gfde$ being analytic \cite{N73}, for example if $\gfde$ has only discrete delays and all coefficients in $\gfde$ are analytic. If the delay is state-dependent, however, this is often not the case. Mallet-Paret and Nussbaum \cite{MPN14} study periodic solutions of time-dependent delay equation of the form
\begin{align}
  \label{mpn:problem}
  v'(t)=f(v(t),v(t-r(t)))
\end{align}
where $f$ and $r$ are analytic. This study is also relevant
for periodic orbits of FDEs with state-dependent delays: if the delay
$r$ has the form $r(t)=r(v(t))$ and \eqref{mpn:problem} has a periodic
orbit $v^*$ of period $T^*$, statements by Mallet-Paret and Nussbaum \cite{MPN14} can be
applied to the problem with time-dependent delay
$r^*(t):=r(v^*(t))$, although it would not be possible to check the assumption on the analyticity of $r^*$ a-priori. Mallet-Paret and Nussbaum \cite{MPN14} observed that the periodic points of the circle map
\begin{align}\label{circle_map}
  t\mapsto t-r(t)\mbox{\quad on the interval $[0,1]$ with periodic extension}
\end{align}
play a role. In unstable periodic points $t_0$ of this circle map the
function $v^*$ will not be analytic, unless additional constraints are
satisfied. Hence, the assumption that $y^*$ is analytic in Theorem~\ref{res:thm:conv} can be checked a-priori for FDEs with constant delays, but can only be indirectly observed by improving convergence properties for increasing $m$ for FDEs with state-dependent delays. M\"uller et al. \cite{muller2018laminar} studied the role of the circle map in the dynamics of systems defined by periodic delays, and observed a previously unknown laminar type of chaos characterized by long plateaus separated by bursts, with plateau levels evolving chaotically. These correspond to attracting periodic points of the inverse circle map.

\section{Convergence of solutions of discretized problem}
\label{sec:convergence}
In this section, we prove that the fixed points of discretized map given in 
\eqref{thm:disc:fp:eq},
\begin{align}
  \Phi_m={J^\mathrm{e}_{10}}\mathcal{P}_mg,\label{def:PhiL}
\end{align}
are locally unique, and converge to the fixed point $x^*$ of $\Phi$. By Corollary~\ref{thm:regularity:fp}, the derivative of $\Phi_m$ at $x^*$,
$D\Phi_m(x^*)={J^\mathrm{e}_{10}}\mathcal{P}_mDg(x^*)$,
is well defined. Similarly to the proof of convergence of the FEM in \cite{asFEM}, we reformulate the discretized fixed
point problem $\Phi_m(x)=x$ in terms of
$\delta^x:=x-x^*$ as
\begin{align}\label{dfp2:split}
  \lefteqn{[I-D\Phi_m(x^*)]\delta^x=}\\
  =&\phantom{+}{J^\mathrm{e}_{10}}(\mathcal{P}_m-I)g(x^*)&&\mbox{(consistency term $\epsilon_\mathrm{c}(m)$)}\label{dfp2:consistency}\\
   &+{J^\mathrm{e}_{10}}\mathcal{P}_m[g(x^*+\delta^x)-g(x^*)-Dg(x^*)\delta^x]&&\mbox{(nonlinearity term $\epsilon_\mathrm{nl}(m,\delta^x)$).}\label{dfp2:nonlinearity}     
\end{align}
Our first goal is to provide a Lipschitz bound for the error $\epsilon_\mathrm{nl}(m,\delta^x)$ in the nonlinear term. Observe that the arguments to bound the nonlinear term in the FEM case \cite[Lemma 6.5]{asFEM} would not be valid, as there is no uniform upper bound for $\|\mathcal{P}_m\|_{L^{\infty}_{\e}\leftarrow C^0_{\e}}$ if $m$ represents the degree of the interpolating polynomial. This is why Assumption \ref{ass:extlip} on extended local Lipschitz continuity of $DG$ (and, hence, $Dg$) is needed. The following lemma follows from Assumption~\ref{ass:extlip}, specifically \eqref{def:extlip:ts}.
\begin{lemma}[Bounds on deviations of $g$]\label{thm:extdev}
  Let Assumption~\ref{ass:extlip} be satisfied in fixed point $x^*$ of
  $\Phi$. Then there exists a radius $r_\mathrm{eL}>0$ and a bound $C_\mathrm{eL}$ such that  for all $\delta^1,\delta^2\in C^{0,1}_\mathrm{e}$ with
  $\|\delta^1\|_{0,1},\|\delta^2\|_{0,1}\leq r_\mathrm{eL}$
  \begin{multline*}
    \left\|g(x^*+\delta^1)-g(x^*+\delta^2)-Dg(x^*)[\delta^1-\delta^2]\right\|_0
    \leq\\ C_\mathrm{eL}\max\left\{\|\delta^1\|_{0,1},\|\delta^2\|_{0,1}\right\}
    \left\|\delta^1-\delta^2\right\|_{0,1}.
  \end{multline*}
\end{lemma}
\begin{proof}
Let $r_\mathrm{eL}$ be the
radius of validity for estimate \eqref{def:extlip:ts}. Let
$\delta^1,\delta^2\in C^1_\mathrm{e}$ with
$\|\delta^1\|_{0,1},\|\delta^2\|_{0,1}\leq r_\mathrm{eL}$. For $C^1_\mathrm{e}$
deviations we can write
\begin{align*}
  \lefteqn{\left\|g(x^*+\delta^1)-g(x^*+\delta^2)-Dg(x^*)[\delta^1-\delta^2]\right\|_0}\\
  =&\left\|\int_0^1\left[Dg(x^*+s\delta^1+(1-s)\delta^2)-Dg(x^*)\right]\left[\delta^1-\delta^2\right]\dd s\right\|_0\\
  \leq&C_\mathrm{eL}\max\left\{\|\delta^1\|_1,\|\delta^2\|_1\right\}
    \left\|\delta^1-\delta^2\right\|_1,
\end{align*}
where the last line applies the extended Lipschitz estimate \eqref{def:extlip:ts} to
deviation $v=s\delta^1+(1-s)\delta^2$ and uses that the convex
combination of $s\delta^1+(1-s)\delta^2$ has a norm bounded by
$\max\left\{\|\delta^1\|_1,\|\delta^2\|_1\right\}$. The bound $C_\mathrm{eL}$ is the same as for \eqref{def:extlip:ts}. Since $g$ and
$Dg(x^*)$ are continuous in $C^0$ arguments and $C^1_\mathrm{e}$ is
dense in $C^{0,1}_\mathrm{e}$ in the $C^0$-norm, we can extend the
estimate to deviations $\delta^1$ and $\delta^2$ in $C^{0,1}_\mathrm{e}$
and the $\|\cdot\|_{0,1}$-norm.\end{proof}
The following lemma provides us
with an estimate for $\epsilon_\mathrm{nl}$.
\begin{lemma}\label{thm:nonlin:geom}
  Let $x^*=(y^*,y^{0,*},\mu^*)$ be a fixed point of $\Phi$ and $G$ be
  mildly differentiable to order $\lm\geq1$, satisfying the extended
  local Lipschitz condition according to
  Assumption~\ref{ass:extlip}. Let the discretization satisfy
  Assumption~\ref{ass:lebesgue} for the mesh with $L$ intervals and
  increasing degrees $m$. Let $r_\mathrm{eL}>0$ as in Lemma~\ref{thm:extdev}, and set $C_\mathrm{nl}=C_\mathrm{eL}\|{J^\mathrm{e}_{10}}\|_{C^{0,1}_\mathrm{e}\leftarrow L^\infty_\mathrm{e}}$ with the bound $C_\mathrm{eL}$ from Lemma~\ref{thm:extdev}.  Then for all $m>0$ and $r\leq r_\mathrm{eL}$
  \begin{align*}
    \|\epsilon_\mathrm{nl}(m,\delta^{x,1})-\epsilon_\mathrm{nl}(m,\delta^{x,2})\|_{0,1}\leq C_\mathrm{nl}\Lambda_m r\|\delta^{x,1}-\delta^{x,2}\|_{0,1}
  \end{align*}
  for all $\delta^{x,1},\delta^{x,2}\in \B^{0,1}_r(0)\subset C^{0,1}_\mathrm{e}$.
  In particular, \textup{(}setting $\delta^x=\delta^{x,1}$ and $\delta^{x,2}=0$\textup{):}
\begin{align*}
  \|\epsilon_\mathrm{nl}(m,\delta^x)\|_{0,1}\leq C_\mathrm{nl}\Lambda_m r\|\delta^x\|_{0,1}\mbox{\quad for all $\delta^x\in B^{0,1}_{r}(0)\subset C^{0,1}_\mathrm{e}$.}
\end{align*}
\end{lemma}
This lemma implies that we can make the Lipschitz constant of
$\delta^x\to\epsilon_\mathrm{nl}(m,\delta^x)$ small if we make the
admissible ball for $\delta^x$ small in the $C^{0,1}$-norm, consistent
with the idea of $\epsilon_\mathrm{nl}$ being the higher-order
term. However, in contrast to the finite-element case, the Lipschitz
constant grows with $m$, such that we may have to make the radius $r$
smaller as we increase $m$.
\begin{proof}
By
Lemma~\ref{thm:extdev} for the radius $r_\mathrm{eL}$  and the constant $C_\mathrm{eL}$ we have
\begin{multline*}
  \left\|g(x^*+\delta^1)-g(x^*+\delta^2)-Dg(x^*)[\delta^1-\delta^2]\right\|_0
  \leq\\ C_\mathrm{eL}\max\left\{\|\delta^1\|_{0,1},\|\delta^2\|_{0,1}\right\}
  \left\|\delta^1-\delta^2\right\|_{0,1}
\end{multline*}
for all $\delta^1,\delta^2\in \B^{0,1}_{r_\mathrm{eL}}(0)$. By definition of $\epsilon_\mathrm{nl}$,
\begin{align*}
  \lefteqn{\|\epsilon_\mathrm{nl}(m,\delta^{x,1})-\epsilon_\mathrm{nl}(m,\delta^{x,2})\|_{0,1}}\\
  &\leq \|{J^\mathrm{e}_{10}}\|_{C^{0,1}_\mathrm{e}\leftarrow L^\infty_\mathrm{e}}\|\mathcal{P}_{m}\|_{ L^\infty_\mathrm{e}\leftarrow C^0_\mathrm{e}}
    \left\|g(x^*+\delta^{x,1})-g(x^*+\delta^{x,2})-Dg(x^*)[\delta^{x,1}-\delta^{x,2}]\right\|_0\\
  &\leq \|{J^\mathrm{e}_{10}}\|_{C^{0,1}_\mathrm{e}\leftarrow L^\infty_\mathrm{e}}\Lambda_m C_\mathrm{eL}\max\left\{\|\delta^{x,1}\|_{0,1},\|\delta^{x,2}\|_{0,1}\right\}
    \left\|\delta^{x,1}-\delta^{x,2}\right\|_{0,1}
\end{align*}
Hence, collecting
\begin{align*}
  C_\mathrm{nl}=\|{J^\mathrm{e}_{10}}\|_{C^{0,1}_\mathrm{e}\leftarrow L^\infty_\mathrm{e}}C_\mathrm{eL}
\end{align*}
and bounding $\|\delta^{x,1}\|_{0,1},\|\delta^{x,2}\|_{0,1}$ by the $C^{0,1}_\mathrm{e}$ radius $r$ results in the claim of the Lemma. 
\end{proof}
A necessary condition
for the well-posedness of the discretized problem is its stability, i.e., the invertibility of the operator
$I-D\Phi_m(x^*)$ uniformly for $m\to\infty$. Thanks to the Banach perturbation lemma, stability follows from the consistency of the derivative, which in turn requires a degree of mild differentiability of at least two (Assumption~\ref{ass:mdiff}).
\begin{lemma}[Consistency of derivative of $\Phi_m$]
\label{thm:stability}
  If $G$ is mildly differentiable twice, then $Dg(x^*)\in\Lin(C^1_\mathrm{e};C^1_\mathrm{e})$, such that $D\Phi_m(x^*)$ satisfies
  \begin{multline}
    \label{rem:DPhi:consistency:sharp}
  \|D\Phi_m(x^*)-D\Phi(x^*)\|_{C^{0,1}_\mathrm{e}\leftarrow C^{0,1}_\mathrm{e}}\leq\\ \frac{3(1+\Lambda_m)}{m}{\|{J^\mathrm{e}_{10}}\|_{C^{0,1}_\mathrm{e}\leftarrow L^\infty_\mathrm{e}}\|Dg(x^*)\|_{C^1_\mathrm{e}\leftarrow C^1_\mathrm{e}}}\cdot \max_{i\leq L}[t_i-t_{i-1}].
  \end{multline}
\end{lemma}
The quantity $\max_{i\leq L}[t_i-t_{i-1}]$ is the maximum of the interval lengths in the fixed grid defining $P_m$ in \eqref{def:interp:proj}.
\begin{proof}
  By mild differentiability of $g$ to order $\ell_{\max}\geq 2$, $x^*\in C^3_\mathrm{e}$ and the map $Dg(x^*)$ is in
  $\Lin(C^1_\mathrm{e};C^1_\mathrm{e})$: $Dg(x^*)v$ exists and is in $C^0_\mathrm{e}$ if $v\in C^0_\mathrm{e}$ and $g$ is mildly differentiable (once). For the time derivative of $Dg(x^*)v$ we get $[Dg(x^*)v]'=D^2g(x^*)[(x^*)']v+Dg(x^*)v'$ if $v\in C^1_\mathrm{e}$, and $g$ is mildly differentiable twice. Since $v\mapsto Dg(x^*)v$ is continuous for $v\in C^0_\mathrm{e}$, we can use the norm $\|Dg(x^*)\|_{C^1_\mathrm{e}\leftarrow C^1_\mathrm{e}}$ also to estimate the Lipschitz constants of $Dg(x^*)v$ for $v\in C^{0,1}_\mathrm{e}$:
  \begin{align*}
      \lip [Dg(x^*)v]\leq \|Dg(x^*)\|_{C^1_\mathrm{e}\leftarrow C^1_\mathrm{e}}\|\|v\|_{0,1}.
  \end{align*}
  Furthermore, the interpolation approximation $P_m$ satisfies for any Lipschitz continuous function $y$ for $h=\max_{i\leq L}[t_i-t_{i-1}]$ (see \cite{riv69})
  \begin{align}
    \|y-P_my\|_\infty \leq 6(1+\Lambda_m)\lip y\frac{h}{2m}.\label{PL:genconv:mod}
  \end{align}
  Thus, 
  \begin{align}
    \|[I-\mathcal{P}_m]Dg(x^*)v\|_\infty \leq 6(1+\Lambda_m)\|Dg(x^*)\|_{C^1_\mathrm{e}\leftarrow C^1_\mathrm{e}}\cdot\frac{h}{2m}
    \mbox{\quad if $\|v\|_{0,1}\leq 1$,}\label{PL:modcont}
  \end{align}
which implies the claim of the lemma after application of operator ${J^\mathrm{e}_{10}}$.
\end{proof}
The stability of the discretized problem follows from Lemma~\ref{thm:stability}, thanks to Assumptions~\ref{ass:lebesgue} and \ref{ass:linear:inv}. We will construct the linear stability constant $C_\mathrm{stab}$ in the proof of Theorem~\ref{thm:conv} below. 

\paragraph{Obstacles for SEM convergence proof in arguments by And{\`o} and Breda \cite{andoSIAM2020}} The approach taken by And{\`o} and Breda \cite{andoSIAM2020} for FEM failed to establish an equivalent of Lemma~\ref{thm:stability}, thus, not providing a route for proving convergence for SEM. Due to the choice by And{\`o} and Breda \cite{andoSIAM2020} of defining the problem on non-periodic spaces of functions, it was necessary to formulate the periodicity condition as the equality between the initial and final \emph{states} (not the \emph{values} at single time points), since otherwise the linearized problem would have not been well-posed (see the comment at the end of \cite[Section 2]{andoSIAM2020}). This meant working with the infinite-dimensional state space, a space of differentiable functions. Given its norm, which involves the derivative of the functions, the consistency of the derivatives of $\Phi_L$ did not hold, which complicated the proof of stability significantly, as it was required to prove the boundedness of $\|I-D\Phi_L(x^*)\|^{-1}$ directly. The expression of the bounds obtained by And{\`o} and Breda \cite{andoSIAM2020} included $\Lambda_m+\Lambda_m'$ as a factor, where
\begin{equation*}
\Lambda'_{m}:=\max_{t\in[t_{i-1},t_{i}]}\sum_{j=\red{1}}^m|l'_{m,j}(t)|
\end{equation*}
for any $i\in\{1,\ldots,L\}$ (see \cite[Section 4.4]{andoSIAM2020} and \cite[equation (A.19) in Proposition A.8]{andoSIAM2020}). The factor $\Lambda_m+\Lambda_m'$ grows unbounded as $m\to\infty$.

\bigskip\noindent
In order to complete the convergence proof we also need a bound for the consistency error $\epsilon_\mathrm{c}(m)$.
\begin{lemma}\label{thm:cons:geom}
  Let $x^*=(y^*,y^{0,*},\mu^*)$ be a fixed point of $\Phi$ and $G$  be mildly differentiable to order at least $\lm\geq 2$. Let
  $(P_{m})_{m=1}^\infty$ be a sequence of projections with slowly diverging Lebesgue constant, satisfying
  Assumption~\ref{ass:lebesgue}. Then there exists a constant
  $C_\mathrm{cs}>0$ such that the consistency error
  $\epsilon_\mathrm{c}(m)$ in  \eqref{dfp2:consistency} satisfies
  \begin{align*}
\|\epsilon_\mathrm{c}(m)\|_{0,1}\leq C_{\mathrm{cs}}\left(\frac{\Lambda_m+1}{m^\lm}\right)\mbox{\quad for all $m\geq\lm-1$.}
\end{align*}
If $y^*$ has an analytic extension to a complex neighborhood of $[0,1]$, there exist constants
  $C_\mathrm{acs},\eta>0$
\begin{align*}
\|\epsilon_\mathrm{c}(m)\|_{0,1}\leq C_\mathrm{acs}\e^{-\eta m}\mbox{\quad for all $m\geq1$.}
\end{align*}
\end{lemma}
\begin{proof}
For the combined solution $x^*=(y^*,y^{0,*},\mu^*)$ we have that $G(y^*_t,\mu^*)=(y^*)'$, and $G$ is the only nonlinear component of $g(x^*)$. Hence, if $x^*\in C^{\ell+1}_\mathrm{e}$, then $g(x^*)\in C^{\ell}_\mathrm{e}$, and if $y^*$ is analytic, so is $G(y^*_t,\mu^*)=(y^*)'$ with the same domain.

By Corollary~\ref{thm:regularity:fp}, $x^*$ is in $C^{\lm+1}_\mathrm{e}$ and, hence, $g(x^*)$ is in $C^{\lm}_\mathrm{e}$. Let $i\in\{0,\ldots,L-1\}$. Then, for $m\geq\lm-1$, $h=\max_{i\leq L}[t_i-t_{i-1}]$,
\begin{equation}\label{epsLu2}
\|(\mathcal{P}_m-I)g(x^*)\vert_{[t_i,t_{i+1}]}\|_{\infty}\leq \|{g(x^{\ast})}^{(\lm)}\|_0(1+\Lambda_{m})c_{\lm}\left(\frac{h}{2m}\right)^{\lm}.
\end{equation}
where $c_{\lm}$ is a positive constant independent of $m$ and $g(x^*)$. \eqref{epsLu2} is a direct consequence of Jackson’s theorem on best uniform
approximation, see, e.g., \cite[(2.9) and (2.11)]{mas15SINA1}.

If $y^*$ is analytic, by Proposition~\ref{thm:geom:app} there exist constants $C_{\mathrm{sp}}$ and $\eta$
such that
\begin{align*}
    \|[I-P_{m}](y^*)'\|_\infty<C_{\mathrm{sp}}\e^{-\eta m}.
\end{align*}
So, overall, choosing 
$$
C_\mathrm{cs}=\left(\frac{h}{2}\right)^{\lm}c_{\lm}\|{g(x^{\ast})}^{(\lm)}\|_0\|{J^\mathrm{e}_{10}}\|_{C^{0,1}_\mathrm{e}\leftarrow L^\infty_\mathrm{e}},\quad C_\mathrm{acs}=C_{\mathrm{sp}}\|{J^\mathrm{e}_{10}}\|_{C^{0,1}_\mathrm{e}\leftarrow L^\infty_\mathrm{e}}
$$
we obtain the claim of the lemma.
\end{proof}
\paragraph{Convergence result}
We can now collect Lemma~\ref{thm:cons:geom} about consistency,
the stability result in Lemma \ref{thm:stability} and the estimate for the
nonlinear part in Lemma~\ref{thm:nonlin:geom} to determine bounds to the convergence rate according to the regularity of $y^*$. The result is the following Theorem~\ref{thm:conv}, which is equivalent to the main result of the paper, Theorem~\ref{res:thm:conv}.
\begin{theorem}[Rate of convergence of collocation]\label{thm:conv}
   Let $G$ be mildly differentiable to order $\lm\geq2$, let
  $x^*=(y^*,y^{0,*},\mu^*)$ be a fixed point of $\Phi$ with bounded inverse of
  $I-D\Phi(x^*)$, and let $G$ be locally Lipschitz continuous in the
  extended sense of Assumption~\ref{ass:extlip}. Let
  the interpolation projection $P_{m}$ have a slowly diverging Lebesgue constant according to Assumption~\ref{ass:lebesgue}.
  For every $\varepsilon>0$ there exists a minimal degree $m_\mathrm{low}$ such that for
  all $m\geq m_\mathrm{low}$ the discretized fixed point problem
  $x=\Phi_{m}(x)$ has a unique solution $x^{m}$ in a ball
  $\B^{0,1}_{r_m}(x^*)$, where the radius $r_m$ depends on $m$. The error
  $\delta^{x,m}=x^{m}-x^*$ satisfies
  \begin{align}
    \label{thm:geom:err}
    \|\delta^{x,m}\|_{0,1}\leq (1+\varepsilon)\|[I-D\Phi(x^*)]^{-1}\|_{0,1}\|\epsilon_\mathrm{c}(m)\|_{0,1}
    \mbox{\quad for all $m\geq m_\mathrm{low}$,}
  \end{align}
where $\|\epsilon_\mathrm{c}(m)\|_{0,1}$ decays according to the regularity of $x^*$, following Lemma~\ref{thm:cons:geom}.
\end{theorem}
\begin{proof}
 Due to the need to override the growth of $\Lambda_m$ as $m\to\infty$, the radius
$r_m$ for the ball in which we have a contraction will have to depend on the
degree $m$.

Similarly to the corresponding proof in the FEM case \cite{asFEM}, consider the map
\begin{align*}
  h(\delta) = [I-D\Phi_{m}(x^*)]^{-1}[\epsilon_\mathrm{c}(m) + \epsilon_\mathrm{nl}(m,\delta)].
\end{align*}
By the identity \eqref{dfp2:split} fixed points of $h$ are fixed points
of $\Phi_m$.
We aim to find $m$ and $r_m$ such that $h$ is a contraction with some rate $\kappa\in(0,1)$. Let the target contraction rate $\kappa\in(0,1)$ be arbitrary, and fix a small $\varepsilon>0$. As by Assumption~\ref{ass:linear:inv} $I-D\Phi(x^*)$ is invertible, we may introduce
\begin{align*}
  C_{\mathrm{stab},\infty}=\|[I-D\Phi(x^*)]^{-1}\|_{0,1},
\end{align*}
and choose the lower bound $m_\mathrm{stab}$ for $m$ such that
\begin{align*}
  C_{\mathrm{stab},m}:=\|[I-D\Phi_{m}(x^*)]^{-1}\|_{0,1}\leq (1+\varepsilon/2)C_{\mathrm{stab},\infty}=:C_\mathrm{stab}
\end{align*}
for all $m\geq m_\mathrm{stab}$, which is possible since, by
Lemma~\ref{thm:stability}, $\|D\Phi_{m}(x^*)-D\Phi(x^*)\|_{0,1}\to0$
for $m\to\infty$. We also choose another lower bound $m_\mathrm{c}$
for the degree $m$ such that, if $y^*$ is not analytic (the constants $C_\mathrm{nl}$, $C_\mathrm{acs}$ and $C_\mathrm{cs}$ below are defined in the proofs of Lemmas \ref{thm:nonlin:geom} and \ref{thm:cons:geom})
\begin{align}\label{proof:mcdef:slow}
  \Lambda_m\cdot\frac{\Lambda_m+1}{m^{\lm}}\leq \frac{(1-\kappa)\kappa}{C_\mathrm{stab}^2C_\mathrm{nl}C_\mathrm{cs}}\mbox{\quad for all $m\geq m_\mathrm{c}$}
\end{align}
while, if $y^*$ is analytic,
\begin{align}\label{proof:mcdef}
  \Lambda_m\e^{-\eta m}\leq \frac{(1-\kappa)\kappa}{C_\mathrm{stab}^2C_\mathrm{nl}C_\mathrm{acs}}\mbox{\quad for all $m\geq m_\mathrm{c}$.}
\end{align}
This is possible since $\lm\geq 2$, by Assumption~\ref{ass:lebesgue}, $\Lambda_m$ grows slower than $m$ and by Proposition~\ref{thm:geom:app} the geometric
convergence rate $\eta$ is positive, such that
$\Lambda_m\e^{-\eta m}\to0$ for $m\to\infty$.

For $m\geq\max\{m_\mathrm{stab},m_\mathrm{c}\}$ we can choose the radius
\begin{align}\label{proof:rmdef}
  r_m&=\frac{\kappa}{C_\mathrm{stab}C_\mathrm{nl}\Lambda_m}
    \mbox{,\quad(thus,  $C_\mathrm{stab}\|\epsilon_\mathrm{c}(m)\|_{0,1}\leq (1-\kappa)r_m$)}
\end{align}
(using \eqref{proof:mcdef:slow}--\eqref{proof:mcdef}) 
for which we plan to show that $h$ maps $\B^{0,1}_{r_m}(0)$ back into
itself and is a contraction with rate $\kappa$.

We have for $\delta^x\in \B^{0,1}_{r_m}(0)$ (using
Lemma~\ref{thm:nonlin:geom} and the definition of $r_m$ in \eqref{proof:rmdef})
\begin{align}
  \nonumber
  \|h(\delta^x)\|_{0,1}&\leq \|[I-D\Phi_m(x^*)]^{-1}\|_{C^{0,1}_\mathrm{e}\leftarrow
  C^{0,1}_\mathrm{e}}\left[\|\epsilon_\mathrm{c}(m)\|_{0,1}+\|\epsilon_\mathrm{nl}(m,\delta^x)\|_{0,1}\right]\\
                       &\leq C_\mathrm{stab}\|\epsilon_\mathrm{c}(m)\|_{0,1}+C_\mathrm{stab}C_\mathrm{nl}\Lambda_m r_m\|\delta^x\|_{0,1}\label{convproof:h:bound}\\
    \nonumber
                       &\leq (1-\kappa)r_m+\kappa \|\delta^x\|_{0,1}\leq r_m.  
\end{align}
Hence, $h$ maps $B^{0,1}_{r_m}(0)$ back into itself. For the Lipschitz constant of $h$ we have by Lemma~\ref{thm:nonlin:geom}
\begin{align*}
  \|h(\delta^{x,1})-h(\delta^{x,2})\|_{0,1}\leq C_\mathrm{stab}C_\mathrm{nl}\Lambda_m r_m\|\delta^{x,1}-\delta^{x,2}\|_{0,1}\leq \kappa \|\delta^{x,1}-\delta^{x,2}\|_{0,1}
\end{align*}
by construction of $r_m$ in \eqref{proof:rmdef}.

Hence, the map $h$ has a unique fixed point $\delta^x$ in $\B^{0,1}_{r_m}(0)$,
which equals a fixed point of $\Phi_{m}$. It satisfies
\begin{align}
  \label{delta:est}
  \|\delta^x\|_{0,1}\leq \frac{C_\mathrm{stab}}{1-\kappa}\|\epsilon_\mathrm{c}(m)\|_{0,1}=\frac{(1+\varepsilon/2)}{1-\kappa}C_{\mathrm{stab},\infty}\|\epsilon_\mathrm{c}(m)\|_{0,1}.
\end{align}
Thus, choosing $\kappa=\varepsilon/(2(1+\varepsilon))$ we obtain the estimate claimed in the theorem.
\end{proof}
\section{Numerical demonstrations}
\label{sec:tests}
The code produced to generate the results presented in this section, and in particular Figures \ref{fig:MG_conv} and \ref{fig:sdquad_conv}, is available at \url{https://github.com/AlessiaWent/sdDDE-SEM-supplementary}.

\noindent
As a first test, we consider the Mackey-Glass equation \cite{MG2010,MG77}
\begin{equation}\label{mackeyglass}
y'(t) = ay(t)+b\cfrac{y(t-\tau)}{1+[y(t-\tau)]^c}
\end{equation}
with $a=-1$, $b=2$ and $c=10$. Under this choice for the parameters, the model undergoes a Hopf bifurcation at $\tau_{\text{Hopf}}\approx 0.4708$. Starting from a sinusoidal perturbation of the equilibrium $\overline y = 1$ as an initial guess, a periodic solution was first found for $\tau=\tau_{\text{Hopf}} + 10^{-3}$ using a scalar affine condition, namely $y(0) = \overline{y}$. Then, the branch of periodic orbits was continued up to $\tau=1$.
We compute the solutions collocating on Gauss-Legendre nodes for different values of $L$ and $m$ and consider the solution $y$ at $\tau=1$. For each test example we \red{measure} the residual of the FDE, divided by the period $T$, \red{on a grid that is much finer  (10001 equidistantly spaced time points in $[0,1]$) than the grids of our discretizations  ($Lm\leq 400$ in all demonstrations). This gives a close approximation of }
\begin{align}
\label{mg:residual}    
\err(y,T,p):=\max_{t\in[0,1]}\left|y'(t)/T-\gfde(y(t+(\cdot)/T),p)\right|\red{=\frac{1}{T}\|(y)'-G(y_{(\cdot)},\mu)\|_0}
\end{align}
\red{in the numerical solution $(y^m,T^m,p^m)$. The quantity $\err_m=\err(y^m,T^m,p^m)$ provides a measure of the residual,
which is not equal to the error $\|\delta^{x,m}\|_{0,1}=\|x^m-x^*\|_{0,1}$ in Theorem~\ref{thm:conv}. The latter is not directly available, as the exact solution of \eqref{mackeyglass} or equation \eqref{eq:quadratic} considered later in this section is not known analytically. However, $\err_m$ and $\|\delta^{x,m}\|_{0,1}$ are related to each other by a factor that is unknown but by the arguments in the proof of Theorem~\ref{thm:conv} asymptotically constant and bounded for $m\to\infty$. Thus, both quantities decay with the same rate. Let us determine this factor. We observe that $\mathcal{P}_mg(x^m)=((y^m)',(y^m)'(0_+),\mu^m)$ and $g(x^m)=(G(y^m,\mu^m),G(y^m,\mu^m)(0),\mu^m)$, such that $\|[\mathcal{P}_m-I]g(x^m)\|_0=T^m\err_m$. 
We can write the decomposition \eqref{dfp2:split}--\eqref{dfp2:nonlinearity} of the error $\delta^{x,m}$ in a slightly different way:
\begin{align}\label{dfp3:split}
  \lefteqn{[I-D\Phi(x^*)]\delta^{x,m}=}\\
  =&\phantom{+}J^\mathrm{e}_{10}(\mathcal{P}_m-I)g(x^m)&&\mbox{(residual term)}\label{dfp3:Err}\\
   &+J^\mathrm{e}_{10}[g(x^*+\delta^{x,m})-g(x^*)-Dg(x^*)\delta^{x,m}]&&\mbox{(nonlinearity term $N(m)$).}\label{dfp3:nonlinearity}     
\end{align}
Lemma~\ref{thm:extdev} ensures that $\|N(m)\|_{0,1}\leq\|J_{10}^\mathrm{e}\|_{C_\pi^1\leftarrow C_\pi^0} C_\mathrm{eL}\|\delta^{x,m}\|_{0,1}^2=C_\mathrm{nl}\|\delta^{x,m}\|_{0,1}^2$ if $\|\delta^{x,m}\|_{0,1}\leq r_\mathrm{eL}$. The $\|\cdot \|_{0,1}$-norm of the inverse of $[I-D\Phi(x^*)]$ is called $C_{\mathrm{stab},\infty}$, such that we can re-arrange \eqref{dfp3:split}--\eqref{dfp3:nonlinearity} as  (assuming that $\|\delta^{x,m}\|_{0,1}\leq r_\mathrm{eL}$)
\[
\|\delta^{x,m}\|_{0,1}\leq C_{\mathrm{stab},\infty}\left[\|J_{10}^\mathrm{e}\|_{C_\pi^{0,1}\leftarrow L_\pi^\infty}T^m\err_m+C_\mathrm{nl}\|\delta^{x,m}\|_{0,1}^2\right].
\]
This implicit estimate implies for errors $\delta^{x,m}$ smaller than the positive constant $\delta_{\max}:=\min\{r_\mathrm{eL},1/(2C_{\mathrm{stab,}\infty}C_\mathrm{nl})\}$
\[
\|\delta^{x,m}\|_{0,1}\leq 2T^mC_{\mathrm{stab},\infty} \|J_{10}^\mathrm{e}\|_{C_\pi^{0,1}\leftarrow L_\pi^\infty}\err_m.
\]}
\red{So, the error $x^m-x^*$ and the residual $\err_m$ are related to each other by the unknown stability constant $C_{\mathrm{stab},\infty}=\|[I-D\Phi(x^*)]^{-1}\|_{0,1}$, apart from the known factors $T^m$, $\|J_{10}^\mathrm{e}\|_{C_\pi^{0,1}\leftarrow L_\pi^\infty}\leq2$, and the safety factor $2$.
Thus we can provide experimental evidence of the convergence rate of $\delta^{x,m}$ by means of the residual $\err_m$, given in \eqref{mg:residual}.} 
Figure \ref{fig:MG_conv} shows the solution profile for the Mackey-Glass equation \eqref{mackeyglass} (on the left), and the residuals $\err_m=\err(y^m,T^m,p^m)$ on the right. The straight lines demonstrate the geometric convergence, and the improving exponent $\eta$ for increasing $L$.
\begin{figure}[t]
\centering
\begin{subfigure}[t]{0.4\textwidth}
    \includegraphics[scale=1.1]{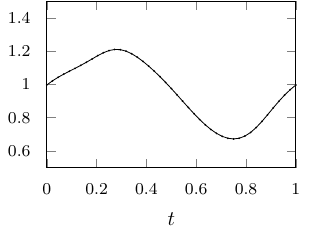}
\end{subfigure}%
~\begin{subfigure}[t]{0.55\textwidth}
\includegraphics[scale=1.1]{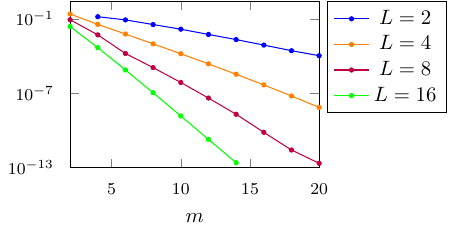}%
\end{subfigure}
\caption{Left: periodic solution computed with $L=11$ and $m=4$, rescaled to period $1$. Right: residual $\err(y,T,\tau)$ for \eqref{mackeyglass}, computed at $10001$ equidistant points in $[0,1]$ for different values of $m$ and $L$, in linear-log scale.}
\label{fig:MG_conv}
\end{figure}
We perform our next test on the BVP
\begin{equation}\label{eq:quadratic}
y'(t) = -y\left(t-\tau-y(t)-y(t)^2\right)
\end{equation}
for $\tau=0.95$ and $\tau=1.1$. As a starting guess for the Newton iterations we choose the solution computed with \textsc{DDE-Biftool}, unadapted mesh with $L=100$ and $m=10$. We recompute the solution collocating on Gauss-Legendre nodes, using different values of $L$ and $m$ and approximate $\err(y,T,\tau)$ along the profiles shown in the left plot of Figure \ref{fig:sdquad_conv}. In contrast to Figure~\ref{fig:MG_conv}, it is not as clear whether the convergence observed is geometric, that is, it is hard to tell whether the lines have a uniform downward slope. Note that the results by Mallet-Paret and Nussbaum \cite{MPN14} are not enough to prove that either of the computed periodic solutions is analytic through any point. Numerically, we observed that the circle map \eqref{circle_map} has periodic points of period $5$ in the case $\tau=0.95$, which are unstable. Mallet-Paret and Nussbaum \cite{MPN14} suggest for this case that the solution obtained is likely not analytic. However, note that the solution is of class $C^{\infty}$ since the coefficients of \eqref{eq:quadratic} are. Thus, by  Theorem \ref{res:thm:conv}, the order of convergence observed in the middle plot of Figure \ref{fig:sdquad_conv} is also compatible with a solution being $C^{\infty}$ but not analytic.
\begin{figure}
~\begin{subfigure}[t]{0.35\textwidth}
\includegraphics[scale=0.9]{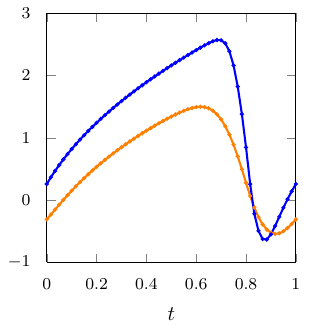}%
\end{subfigure}
\begin{subfigure}[t]{0.38\textwidth}
    \includegraphics[scale=0.9]{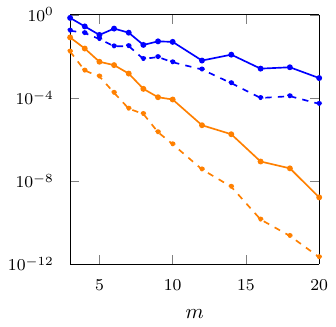}
\end{subfigure}
~\begin{subfigure}[t]{0.2\textwidth}
\includegraphics[scale=0.9]{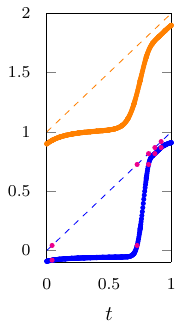}%
\end{subfigure}%
\caption{Left: Periodic solution of BVP defined by \eqref{eq:quadratic} at $\tau = 0.95$ (blue) and $\tau = 1.1$ (orange) computed with $L=12$ and $m=5$, rescaled to period $1$. Middle: $\err(y,T,\tau)$ computed at $10001$ equidistant points in $[0,1]$ for $L=10$ (solid) and $L=20$ (dashed), in linear-log scale. Right: first iterates of the circle map \eqref{circle_map} modulo $[0,1]$. The five fixed points of the fifth iterate for $\tau=0.95$ are shown in magenta.}
\label{fig:sdquad_conv}
\end{figure}

\section{Concluding remarks}
\label{s_concluding}
Collocation on piecewise orthogonal polynomials is used for computation of periodic solutions of FDEs in publicly available libraries such as \textsc{DDE-Biftool} \cite{ED02,ELHR01,ELR02,ddebiftoolmanual}, \textsc{knut} \cite{RS07,SSH06} and \textsc{coco} \cite{ahsan2022methods,DS13}. While the convergence analysis of the FEM strategy in this context has already been completed for constant delay by And{\`o} and Breda \cite{andoSIAM2020} and state-dependent delays by And{\`o} and Sieber \cite{asFEM}, the question concerning the convergence of the spectral element method (SEM) had been open, even in the case of constant delay. The present paper bridges this gap by proving that SEM converges with a rate depending on the regularity of the periodic solution. In particular SEM converges geometrically if the solution is analytic.

State-dependent delays often appear in renewal (Volterra integral) equations in the form of threshold delays in population growth models \cite{diekmann2010daphnia,gedeon2022operon}. The present analysis for proving higher-order convergence for periodic BVPs may also be applicable to renewal equations, either for FEM or for SEM.

\section*{Acknowledgments}
A.\,A.\ is a member of INdAM Research group GNCS. This work was supported by the Italian Ministry of University and Research (MUR) through the PRIN 2020 project (No. 2020JLWP23) “Integrated Mathematical Approaches to Socio-Epidemiological Dynamics”, Unit of Udine (CUP G25F22000430006). The research collaboration was supported by the Lorentz Center Leiden (The Netherlands) workshop ``\emph{Towards rigorous results in state-dependent delay equations}'',
4--8 March 2024.
\bibliographystyle{abbrvnat}
\bibliography{delay}

\end{document}

\begin{equation}\label{}
\left\{
\begin{aligned}
A &= (-1)^k \frac{4}{k\pi} \sin \left(\frac{k \pi}{2}\right) \frac{\gamma}{2} (1-2\theta) A, \\
\theta &= \gamma \theta (1-\theta) -\frac{\gamma}{2}A^{2}.
\end{aligned}
\right.
\end{equation}

\begin{align}
\mathcal{K}(\alpha_{1},\alpha_{2})&\coloneqq e^{\int_{\alpha_{2}}^{\alpha_{1}}g_{1}(\theta)d\theta}g_{2}(\alpha_{2}),\label{K}\\
\mathcal{Klambda}(\alpha_{1},\alpha_{2})&\coloneqq -\mathcal{F}(\alpha_{1},\bar{TB})\left(\int_{\alpha_{2}}^{\alpha_{1}}\mu_{1}(\theta)\mathcal{K}(\theta,\alpha_{2})d\theta+\mu_{2}(\alpha_{2})\right).\label{Klambda}
\end{align}

\begin{equation*}
\begin{split}
z(t+4)={}& \int_{0}^4 A(\tau) h(z(t+4-\tau)) {\rm d} \tau \\
={}& \int_{0}^4 A(4-\sigma)h(z(t+\nu)){\rm d} \sigma \\
={}& \int_{0}^4 A(\tau) h(z(t+\tau)) {\rm d} \tau \\
={}& \int_{0}^4 A(\tau) h(z(-t-\tau)) {\rm d} \tau \\
={}& z(-t),
\end{split}
\end{equation*}

\begin{equation}\label{equation1}
\begin{split}
x'&=x(\alpha-\beta y),\\
y'&=-y(\gamma-\delta x).
\end{split}
\end{equation}

\begin{equation}\label{equation2}
\begin{split}
\dot{y}(t) ={}& a y(t)+b f(y(t-\gamma_{0}\tau)))\\
>{}& a y(t)-b f(y(t-\gamma_{0}\tau))\\
& +b f(y(t-\gamma_{0}\tau))\\
={}& a y(t).
\end{split}
\end{equation}

\begin{align}
\dot{y}(t) &= a y(t)+b f(y(t)),\label{equation3a}\\
\ddot{y}(t) &= a y(t) - b f(y(t)),\label{equation3b}\\
y(0) &= y_{0}.\label{equation3c}
\end{align}

\begin{equation}\label{equation4}
|x|=\begin{cases}
x,\quad &x \geq 0,\\
-x, \quad &x<0.
\end{cases}
\end{equation}

\begin{equation*}
\begin{split}
\alpha x(t)+\beta y(t)= {} &\int_{t_{0}}^{t}f(x(s),y(s-\tau))\,ds-\int_{t_{0}}^{t}g(x(s-\tau),y(s))\,ds\\
&+\int_{t_{0}}^{t}h(x(s-\tau),y(s-\tau))\,ds.
\end{split}
\end{equation*}

\begin{figure}[!h]
\centering
\includegraphics[width=\textwidth]{}
\caption{caption. See text for more details.}\label{f_E}
\end{figure}

\section{}
\label{s_}
\subsection{}
\label{s_}

{\color{red} ()}
